\documentclass{amsart}   	
\usepackage{geometry}                		
\usepackage{graphicx}				
\usepackage{amssymb}
\usepackage{amsmath}
\usepackage{amsfonts}
\usepackage{mathrsfs}
\usepackage{amsthm}
\usepackage{color}
\usepackage{dsfont}
\usepackage{tikz}
\usepackage{tikz-cd}
\usepackage[all]{xy}
\usepackage{mathtools}
\tikzcdset{column sep/my size/.initial=0.1ex}

\def\an{\mathrm{an}}

\def\comp{\mathrm{comp}}
\def\cts{\mathrm{cts}}

\def\id{\mathrm{id}}
\def\un{\mathrm{un}}
\def\zag{{\mathrm{Zag}}}
\def\alg{{\mathrm{alg}}}
\def\naive{\mathrm{naive}}
\def\KZB{{\mathrm{KZB}}}

\def\reg{{\mathrm{reg}}}

\def\SL{{\mathrm{SL}}}
\def\Lie{{\mathrm{Lie}}}
\def\Der{{\mathrm{Der}}}
\def\DR{{\mathrm{dR}}}
\def\B{\mathrm B}

\def\bone{\mathds{1}}
\def\dbs{\backslash\hspace{-.04in}\backslash}
\def\ddq{{\partial /\partial q}}
\def\h{\mathfrak h}
\def\ll{\langle\langle}
\def\rr{\rangle\rangle}
\def\p{{\mathfrak p}}

\def\a{\mathbf{a}}
\def\bb{\mathbf{b}}
\def\adual{\check{\mathbf a}}
\def\bdual{\check{\mathbf b}}
\def\e{\mathbf{e}}
\def\u{\mathfrak{u}}
\def\vv{{\vec{\mathsf v}}}

\def\w{\mathbf{w}}

\def\AA{\mathsf A}
\def\C{\mathbb C}
\def\Ca{\mathcal C}
\def\D{\mathbb D}

\def\Eis{G}
\def\G{\mathbb G}

\def\H{\mathbb H}
\def\cH{\mathcal H}
\def\Hbar{\overline{\mathcal H}}
\def\T{\mathsf{T}}
\def\Ss{\mathsf{S}}
\def\E{\mathcal E}
\def\Ebar{\overline{\mathcal E}}
\def\Li{{\mathrm{Li}}}
\def\LL{{\mathbb L}}

\def\K{\mathbb K}
\def\M{\mathcal{M}}
\def\Mbar{\overline{\mathcal M}}

\def\OO{\mathcal O}
\def\P{\mathbb P}
\def\cP{\mathcal P}
\def\Pbar{\overline{\mathcal P}}
\def\Q{\mathbb Q}

\def\U{\mathcal U}
\def\V{\mathbb V}
\def\cV{\mathcal V}
\def\Vbar{\overline{\cV}}
\def\XX{\mathscr X}
\def\Z{\mathbb Z}

\newtheorem*{theorem*}{Theorem}
\newtheorem*{corollary*}{Corollary}
\newtheorem{theorem}{Theorem}[section]
\newtheorem{proposition}[theorem]{Proposition}
\newtheorem{corollary}[theorem]{Corollary}
\newtheorem{lemma}[theorem]{Lemma}

\theoremstyle{definition}

\newtheorem{example}[theorem]{Example}

\newtheorem{formula}[theorem]{Formula}

\theoremstyle{remark}
\newtheorem{remark}[theorem]{Remark}

\newcommand\ad{\operatorname{ad}}

\newcommand\Aut{\operatorname{Aut}}
\newcommand\End{\operatorname{End}}
\newcommand\Hom{\operatorname{Hom}}
\newcommand\Ext{\operatorname{Ext}}
\newcommand\im{\operatorname{im}}
\newcommand\Spec{\operatorname{Spec}}
\font \rus= wncyr10
\newcommand{\sha}{\, \hbox{\rus x} \,}
\newcommand*{\DashedArrow}[1][]{\mathbin{\tikz [baseline=-0.25ex,-latex, dashed,#1] \draw [#1] (0pt,0.5ex) -- (1.3em,0.5ex);}}%

\title[Algebraic De Rham Theory for Unipotent Fundamental Groups of Elliptic Curves]{The Elliptic KZB Connection and Algebraic De Rham Theory for Unipotent Fundamental Groups of Elliptic Curves}
\author{Ma Luo}
\date{\today}							
\thanks{This paper is mostly written at Duke University, and then revised at University of Oxford.}

\begin{document}
\maketitle

\tableofcontents


\section*{Introduction}
A once punctured elliptic curve is an elliptic curve with its identity removed. In this paper, we describe an algebraic de Rham theory for unipotent fundamental groups of once punctured elliptic curves by explicit computations. 

The analytic version of this story is described by Calaque, Enriquez and Etingof \cite{cee} and by Levin and Racinet \cite{levin-racinet}. For a family $\XX\to T$ where each fiber $X_t$ over $t\in T$ is a once punctured elliptic curve,
\begin{center}
\begin{tikzcd}[column sep=my size]
X_t \arrow[d, no head, no tail] & \subset & \XX \arrow[d, no head, no tail]\\
t & \in & T
\end{tikzcd}\\
\end{center}
there is a vector bundle $\cP$ of prounipotent groups over $\XX$, endowed with a flat connection. For a point $x\in\XX$ that lies over $t$, i.e. $x\in X_t$, the fiber of $\cP$ over $x$ is the unipotent fundamental group $\pi_1^\un(X_t,x)$. We are particularly interested in the case when $\XX=\E'$ and $T=\M_{1,1}$. In this case, the flat connection is called the universal elliptic KZB\footnote{Named after physicists Knizhnik, Zamoldchikov and Bernard.} connection\footnote{The general universal elliptic KZB connection is the flat connection on the bundle $\cP$ over $\M_{1,n+1}$ whose fiber over $[E;0,x_1,\dots,x_n]$ is the unipotent fundamental group of the configuration space of $n$ points on $E'$ with base point $(x_1,\dots,x_n)$. Calaque et al \cite{cee} write down the universal elliptic KZB connection for all $n\ge 1$.}. The bundle $\cP$ extends naturally by Deligne's canonical extension $\Pbar$ to $\Ebar$, and the universal elliptic KZB connection on $\Pbar$ has regular singularities around boundary divisors: the identity section of the universal elliptic curve $\E\to\M_{1,1}$, and the nodal cubic. One can restrict the bundle $\cP$ to a single fiber of $\E'\to\M_{1,1}$, i.e. a once punctured elliptic curve $E':=E-\{\id\}$, and obtain Deligne's canonical extension $\Pbar$ of it to $E$. It is endowed with a unipotent connection on $E$, having regular singularity at the identity. We call this the elliptic KZB connection on $E$.

Working algebraically, Levin and Racinet have shown that there is a $\K$-structure on the bundle $\cP$ over a punctured elliptic curve $X$ defined over a field $\K$ of characteristic zero, and a $\Q$-structure for the bundle $\cP$ over $\M_{1,2/\Q}$. However, their algebraic formulas for the (universal) elliptic KZB connection is neither explicit nor having regular singularity at the identity (section). By resolving these issues for the case of $\M_{1,2}$, we will prove that
\begin{theorem*}
There is an explicit $\Q$-de Rham structure $\Pbar_\DR$ on the bundle $\Pbar$ over $\Mbar_{1,2}$ with the universal elliptic KZB connection, which has regular singularities along boundary divisors, the identity section and the nodal cubic.
\end{theorem*}
Restricting to a single elliptic curve, we get
\begin{corollary*}
If $E$ is an elliptic curve defined over a field $\K$ of characteristic zero, then there is an explicit $\K$-de Rham structure $\Pbar_\DR$ on the bundle $\Pbar$ over $E$ with an elliptic KZB connection, which has regular singularity at the identity.
\end{corollary*}
It is well known that these de Rham structures exist, and can be constructed via Tannaka duality. But the explicit nature allows us to compute the image of these de Rham structures on the complexification of the Betti ones.

\bigskip

In the classical case of $\P^1-\{0,1,\infty\}$, there is a trivial vector bundle $\cP$ on it with fiber $\C\ll\e_0,\e_1\rr$, formal power series in non-commuting generators $\e_0$, $\e_1$, and with the KZ connection
$$\nabla=d-\frac{dz}{z}\e_0+\frac{dz}{1-z}\e_1,$$
where $\e_0$, $\e_1$ act on the fiber by left multiplication. This bundle extends via Deligne's canonical extension to a trivial bundle $\Pbar$ over $\P^1$ with the same connection having regular singularities at $0,1,\infty$. One can view this bundle as a universal object of the tannakian category
$$\Ca_\Q:=\left\{
\begin{tabular}{c}
unipotent vector bundles $\Vbar$ over $\P^1$ defined over $\Q$\\
with a flat connection $\nabla$ that has regular \\
singularities at $0,1,\infty$ with nilpotent residue
\end{tabular}\right\},$$
and interpret the KZ connection as a universal unipotent connection on $\P^1_{/\Q}$. Moreover, classical polylogarithms 
$$\Li_k(z):=\sum_{n>0}\frac{z^n}{n^k},\qquad k\ge 1$$
and their generalizations
$$\Li_{k_1,\cdots,k_r}(z):=\sum_{n_1>\cdots>n_r>0}\frac{z^{n_1}}{n_1^{k_1}\cdots n_r^{k_r}},\qquad r>0, k_i>0,$$
interpreted \cite{beilinson-deligne} as periods of unipotent variations of mixed Hodge structures on $\P^1-\{0,1,\infty\}$, can be expressed as (regularized) iterated integrals of algebraic 1-forms $\frac{dz}{z}$ and  $\frac{dz}{1-z}$ that appear in the KZ connection.

It is important to note from \cite[\S 12]{deligne:p1} that in the ``good" case when a smooth variety $X$ is defined over a field $\K$, and its compactification $\overline X$ satisfies the conditions of 
$$H^0(\overline{X},\OO)=\K \quad\text{and}\quad H^1(\overline{X},\OO)=0,$$
Deligne's canonical extension to $\overline X$ of a unipotent vector bundle $\cV$ over $X$, $\Vbar$, is a trivial bundle. Clearly this works for $X=\mathbb P^1-\{0,1,\infty\}$, and allows one to define a {\em global fiber functor} on the tannakian category $\Ca_\Q$ of unipotent connections over $\P^1_{/\Q}$ with regular singularities at $0,1,\infty$. However, this does not work for punctured elliptic curves.

The tricky part in our elliptic case is that Deligne's canonical extension of $\cP_\DR$, $\Pbar_\DR$, is not a trivial bundle as the corresponding monodromy representation fails a Hodge theoretic restriction (see \cite{hilbert}). We trivialize the bundle $\Pbar_\DR$ on two open subsets of the (universal) elliptic curve, and write down algebraic connection formulas according to these trivializations, with suitable gauge transformation on their intersection. These two opens are $E'$ and $E''$ (correspondingly, $\E'$ and $\E''$ for the universal elliptic curve $\E$), where $E''$, containing the identity, is the complement of three non-trivial 2-torsion points of $E$ (the trivial one being the identity).

For a single elliptic curve $E_{/\K}$, this bundle $\Pbar_\DR$ is a universal object of the tannakian category
$$\Ca^\DR_\K:=\left\{
\begin{tabular}{c}
unipotent vector bundles $\Vbar$ over $E$ defined over $\K$\\
with a flat connection $\nabla$ that has regular singularity\\
at the identity with nilpotent residue
\end{tabular}\right\}.$$
This allows us to compute the tannakian fundamental group $\pi_1(\Ca^\DR_\K,\omega)$ of this category, where the fiber functor $\omega$ is the fiber over a $\K$-rational point of $E$. It is a free pronilpotent group of rank 2 defined over $\K$. This tannakian formalism implies that the elliptic KZB connection that we have computed explicitly on the algebraic vector bundle $\Pbar_\DR$ can be viewed as a universal unipotent connection over $E_{/\K}$. Similar results have been recently obtained by Enriquez--Etingof \cite{ee} for the configuration space of $n$ points in an elliptic curve $E$ with ground field $\C$. 

The elliptic KZB connection enables us to construct closed\footnote{i.e. homotopy invariant.} iterated integrals of algebraic 1-forms of any lengths on $E$, so that one can compute the (regularized) periods of $\pi_1^\un(E',\vv)$, where $\vv$ is a tangential base point at the identity. This leads naturally to elliptic polylogarithms \cite{beilinson-levin,brown-levin,levin-racinet}. Note that it is not known in general how to construct closed iterated integrals of algebraic 1-forms of lengths 3 or more. In Section \ref{naive}, we show that the naive elliptic analogue of the KZ connection
$$\nabla=d-\frac{xdx}{y}\T+\frac{dx}{y}\Ss,$$
 which we call the naive connection, is somewhat different than the elliptic KZB connection, see Proposition \ref{Kstr} for the precise statement. However, one can show that the elliptic KZB connection on $E$ is algebraically gauge equivalent to the naive connection up to degree 5, see Remark \ref{age}. Therefore, periods constructed from both connections, as iterated integrals of algebraic 1-forms of lengths at most 5, are the same. 

\bigskip

\noindent{\em Acknowledgments:} I would like to thank my advisor, Richard Hain, for introducing me this problem and providing numerous suggestions and discussions in writing this paper. I would like to thank Francis Brown for discussions, especially on Section \ref{algconn}. I would also like to thank Jin Cao, Benjamin Enriquez, Minhyong Kim, Nils Matthes and Alex Saad for their interest in the paper. Finally, I would like to thank the referees for their thorough reading and useful comments on the manuscript.

\part{Background}

\section{Moduli Spaces of Elliptic Curves}

Here we quickly review the construction of moduli spaces of elliptic curves and their Deligne-Mumford compactifications. Full details can be found in \cite{km}, \cite{elliptic}, or the first section of \cite{kzb}.

\subsection{Moduli spaces as algebraic stacks}\label{stack}
Denote the moduli stack over $\Q$ of elliptic curves with $n$ marked points and $r$ non-zero tangent vectors by $\M_{1, n+\vec{r}}$. 
The Deligne-Mumford compactification of $\M_{1,n}$ will be denoted by $\Mbar_{1,n}$. 

One can view $\M_{1,n+1}$ as the stack quotient of $\M_{1,n+\vec{1}}$ by the $\G_m$-action:
$$
\lambda: [E;x_1,\cdots,x_n;\omega]\mapsto [E;x_1,\cdots,x_n;\lambda\omega],
$$
where a moduli point $[E;x_1,\cdots,x_n;\omega]\in\M_{1,n+\vec{1}}$ is represented by tuples 
$$(E;x_1,\cdots,x_n;\omega),$$
an elliptic curve $E$ with $n$ marked points and the differential form $\omega$ that is dual to the marked tangent vector.

For example, the moduli space $\mathcal{M}_{1,\vec{1}}$ over $\Q$ can be described explicitly as the scheme
$$\mathcal M_{1,\vec{1}}=\mathbb A^2_{\Q}-D,$$
where $D$ is the discriminant locus $\{(u,v)\in\mathbb A^2:\Delta=u^3-27v^2= 0\}$, see \cite[\S 2.2]{km}. The point $(u,v)$ corresponds to the once punctured elliptic curve (the plane cubic) $y^2=4x^3-ux-v$ with the abelian differential $dx/y$. The moduli stack $\mathcal M_{1,1}$ can be defined as the quotient of $\mathcal M_{1,\vec{1}}$ by the $\mathbb G_m$-action:
$$\lambda\cdot(u,v)=(\lambda^{-4}u,\lambda^{-6}v).$$
Its compactification, the moduli stack $\Mbar_{1,1}$, is the quotient of $Y:=\mathbb A^2-\{(0,0)\}$\footnote{One may regard $Y$ as a partial compactification of $\M_{1,\vec{1}}$.} by the same $\G_m$-action above. 

Similarly, define the moduli stack $\mathcal{M}_{1,2}$ over $\Q$ to be the quotient of the scheme
$$\mathcal M_{1,1+\vec{1}}:=\{(x,y,u,v)\in\mathbb{A}^2\times\mathbb{A}^2: y^2=4x^3-ux-v, \text{ and }u^3-27v^2\neq 0 \}$$
by the $\mathbb{G}_m$-action
$$\lambda:(x,y,u,v)\mapsto(\lambda^{-2}x,\lambda^{-3}y,\lambda^{-4}u,\lambda^{-6}v).$$
Here the point $(x,y,u,v)\in\M_{1,1+\vec{1}}$ corresponds to the point $(x,y)$ on the punctured elliptic curve $y^2=4x^3-ux-v$ with the abelian differential $dx/y$. Note that $\M_{1,2}$ is $\E'$, the universal elliptic curve $\E$ over $\M_{1,1}$ with its identity section removed. We define its compactification $\Mbar_{1,2}$ as the quotient of the scheme
$$\{(x,y,u,v)\in\mathbb{A}^2\times\mathbb{A}^2: y^2=4x^3-ux-v, (u,v)\neq (0,0) \}$$
by the same $\G_m$-action above. Note that $\Mbar_{1,2}$ is the compactification $\Ebar$ of the universal elliptic curve $\E$ whose restriction to the $q$-disk is the Tate curve $\E_{\mathrm{Tate}}\to\Spec\Z[[q]]$ (see \cite[Chap. V]{silverman}).

\subsection{Moduli spaces as complex analytic orbifolds}\label{orbi}
Working complex analytically, we can define moduli spaces as complex orbifolds
$$
\M_{1,1}^\an:=\G_m\dbs\M_{1,\vec{1}}^\an, \quad \M_{1,2}^\an:=\G_m\dbs\M_{1,1+\vec{1}}^\an,
$$
where $\M_{1,\vec{1}}^\an:=\M_{1,\vec{1}}(\C)$ and $\M_{1,1+\vec{1}}^\an:=\M_{1,1+\vec{1}}(\C)$ are complex analytic manifolds.

The moduli space $\M_{1,1}^\an$ can also be defined as the orbifold quotient $\SL_2(\Z)\dbs\h$ of the upper half plane $\h$ by the standard $\SL_2(\Z)$ action:
$$\gamma=\begin{pmatrix} a & b \cr c & d\end{pmatrix}:\tau\mapsto \gamma\tau=\frac{a\tau+b}{c\tau+d}$$
where $\gamma\in \SL_2(\Z)$ and $\tau\in\h$. The map
\begin{align*}
\h &\to \M_{1,\vec{1}}^\an=\{(u,v)\in\C^2: u^3-27v^2\neq 0\}
\\ \tau &\mapsto (20\Eis_4(\tau),\frac 73\Eis_6(\tau))
\end{align*}
induces an isomorphism of orbifolds $\SL_2(\Z)\dbs\h\cong\G_m\dbs\M_{1,\vec{1}}^\an$, where $\Eis_{2n}(\tau)$ is the normalized Eisenstein series of weight $2n$ (see Section \ref{eis} for a definition).

A point $\tau\in\h$ corresponds to the framed elliptic curve $E_{\tau}:=\C/\Lambda_{\tau}$ where $\Lambda_{\tau}:=\Z\oplus\Z\tau$, with a basis $\a, \bb$ of $H_1(E_{\tau};\Z)$ that corresponds to $1,\tau$ of the lattice $\Lambda_{\tau}$ via the identification $H_1(E_{\tau};\Z)\cong\Lambda_{\tau}$.

There is a canonical family of elliptic curves $\E_{\h}$ over the upper half plane $\h$, called the universal framed family of elliptic curves in \cite{elliptic}, whose fiber over $\tau\in\h$ is $E_{\tau}$. It is the quotient of the trivial bundle $\C\times\h\rightarrow\h$ by the $\Z^2$-action:
$$(m,n) : (\xi,\tau) \mapsto \bigg(\xi + \begin{pmatrix} m & n \end{pmatrix}
\begin{pmatrix} \tau \cr 1 \end{pmatrix}, \tau \bigg).$$

The universal elliptic curve $\E^\an$ over $\M_{1,1}^\an$ is the orbifold quotient of $\C\times\h$ by the semi-direct product\footnote{The semi-product structure is induced from the right action of $\SL_2(\Z)$ on $\Z^2$:
$\begin{pmatrix} a & b \cr c & d\end{pmatrix} :
\begin{pmatrix} m & n \end{pmatrix} \mapsto
\begin{pmatrix}m & n\end{pmatrix}\begin{pmatrix} a & b \cr c & d\end{pmatrix}.$} $\SL_2(\Z) \ltimes \Z^2$,
where $\Z^2$ acts on $\C\times\h$ as above, and $\gamma\in \SL_2(\Z)$ acts as follows:
$$\gamma : (\xi,\tau) \mapsto\big((c\tau+d)^{-1}\xi,\gamma \tau\big).$$
The universal elliptic curve $\E^\an$ can also be obtained as the orbifold quotient of $\E_{\h}\rightarrow\h$ by $\SL_2(\Z)$. It is an orbifold family of elliptic curves whose fiber over a moduli point $[E]\in\M_{1,1}$ is an elliptic curve isomorphic to $E$. If we remove all the lattice points $\Lambda_\h:=\{(\xi,\tau)\in\C\times\h:\xi\in\Lambda_\tau\}$ from $\C\times\h$, then take the orbifold quotient of the same $\SL_2(\Z) \ltimes \Z^2$-action above, we obtain another analytic description of the moduli space $\M_{1,2}^\an$. To relate the two descriptions, there is a map
\begin{align*}
(\C\times\h)-\Lambda_\h &\to \M_{1,1+\vec{1}}^\an=\{(x,y,u,v)\in\C^2\times\C^2: y^2=4x^3-ux-v, (u,v)\neq (0,0) \}
\\ (\xi,\tau) &\mapsto (P_2(\xi,\tau), -2P_3(\xi,\tau),20\Eis_4(\tau),\frac 73\Eis_6(\tau))
\end{align*}
that induces an isomorphism $\SL_2(\Z) \ltimes \Z^2\dbs((\C\times\h)-\Lambda_\h)\cong\G_m\dbs\M_{1,1+\vec{1}}^\an$, where $P_2(\xi,\tau)$ and $P_3(\xi,\tau)$ are, up to a constant, the Weierstrass $\wp$-function $\wp_\tau(\xi)$ and its derivative $\wp'_\tau(\xi)$ (see Section \ref{eisell} for the definition).


\section{The Local System $\H$ with its Betti and $\Q$-de Rham Realizations}

The local system $\H$ over $\M_{1,1}$ is a ``motivic local system", which has a set of compatible realizations: Betti, $\Q$-de Rham, Hodge and $l$-adic described in \cite[\S 5]{umem}. In this section we will follow \cite[\S 5]{umem} closely and describe its Betti realization $\H^\B$ and $\Q$-de Rham realization $\cH^\DR$, and the comparison between these two. 

We will denote the pull back of $\H$ (resp. $\H^\B$, $\cH^\DR$) to $\M_{1,n+\vec{r}}$ by $\H_{n+\vec{r}}$ (resp. $\H^\B_{n+\vec{r}}$, $\cH^\DR_{n+\vec{r}}$), so that $\H_1$ (resp. $\H^\B_1$, $\cH^\DR_1$) is the same as $\H$ (resp. $\H^\B$, $\cH^\DR$).

\subsection{Betti realization $\H^\B$}\label{locsys}
The Betti realization $\H^\B$ is the local system $R^1\pi^\an_\ast\Q$ over $\M_{1,1}^\an$ associated to the universal elliptic curve $\pi^\an: \E^\an \to \M_{1,1}^\an$. We identify it, via Poincar\'e duality $H^1(E) \to H_1(E)$ fiberwise, with the local system over $\M_{1,1}^\an$ whose fiber over $[E] \in \M_{1,1}$ is $H_1(E;\Q)$.

There is a natural $\SL_2(\Z)$ action
$$
\gamma=\begin{pmatrix} a & b \cr c & d\end{pmatrix}: \begin{pmatrix} \bb \cr \a\end{pmatrix}\mapsto\begin{pmatrix} a & b \cr c & d\end{pmatrix}\begin{pmatrix} \bb \cr \a\end{pmatrix},
$$
where $\a,\bb$ is the basis of $H_1(E_\tau;\Z)$ that corresponds to the basis $1,\tau$ of $\Lambda_\tau$. The sections $\a,\bb$  trivialize the pullback $\H_\h$ of $\H^\B$ to $\h$.

Denote the dual basis of $H^1(E_\tau;\Q)\cong \Hom(H_1(E_\tau),\Q)$ by $\adual,\bdual$. Then, under Poincar\'e duality,
$$
\adual = -\bb \text{ and } \bdual = \a.
$$
And the corresponding $\SL_2(\Z)$ action on this dual basis is
$$
\gamma:\begin{pmatrix} \a & -\bb\end{pmatrix}\mapsto\begin{pmatrix} \a & -\bb\end{pmatrix}\begin{pmatrix} a & b \cr c & d\end{pmatrix}.
$$
One can construct the local system $\H^\B$ by taking the orbifold quotient of the local system $\H_\h\to\h$ by the above $\SL_2(\Z)$-action.

\subsection{$\Q$-de Rham realization $\cH^\DR$}
The $\Q$-de Rham realization $\cH^\DR$ is an algebraic vector bundle $H^1(\E/\M_{1,1})$ over $\Q$, defined by relative cohomology, on $\M_{1,1/\Q}$ equipped with the Gauss-Manin connection. By our definition $\M_{1,1}:=\G_m\dbs\M_{1,\vec{1}}$, to work with $\M_{1,1}$ is to work $\G_m$-equivariantly with $\M_{1,\vec{1}}$. We first construct the corresponding algebraic vector bundle on $\M_{1,\vec{1}}$. Define a vector bundle
$$
\cH_{\vec{1}}^\DR:=\OO_{\M_{1,\vec{1}}}\Ss\oplus\OO_{\M_{1,\vec{1}}}\T
$$
over $\M_{1,\vec{1}}$ with a $\G_m$-action:
$$
\lambda\cdot\Ss=\lambda^{-1}\Ss\quad\text{and}\quad\lambda\cdot\T=\lambda\T,
$$
where the sections $\Ss$ and $\T$ represent algebraic differential forms $xdx/y$ and $dx/y$ respectively. This $\G_m$-action extends the action on $\M_{1,\vec{1}}$ to the bundle $\cH_{\vec{1}}^\DR$ over it.

The Gauss-Manin connection is explicitly given by
\begin{equation}
\nabla_0=d+
\left(-\frac{1}{12}\frac{d\Delta}{\Delta}\T+\frac{3\alpha}{2\Delta}\Ss\right)\frac{\partial}{\partial \T}+\left(-\frac{u\alpha}{8\Delta}\T+\frac{1}{12}\frac{d\Delta}{\Delta}\Ss\right)\frac{\partial}{\partial \Ss},
\end{equation}
where $\alpha=2udv-3vdu$, $\Delta=u^3-27v^2$ and $\frac{\partial}{\partial\Ss}$,$\frac{\partial}{\partial\T}$ a dual basis for $\Ss$,$\T$ (cf. \cite{kzb} Prop 19.6). This connection is $\G_m$-invariant, and defined over $\Q$. Therefore, the bundle $\cH^\DR_{\vec{1}}$ with connection $\nabla_0$ over $\M_{1,\vec{1}}$ descends to a bundle $\cH^\DR$ over $\M_{1,1}$.

The canonical extension $\Hbar^\DR_{\vec{1}}$ of $\cH^\DR_{\vec{1}}$ to $Y:=\mathbb A_\Q^2-\{(0,0)\}$ is a vector bundle 
$$
\Hbar^\DR_{\vec{1}}:=\OO_Y\Ss\oplus\OO_Y\T
$$ 
with the same connection $\nabla_0$ above. Since the connection has regular singularities along the discriminant locus $D=\{\Delta=0\}$, and recall that $\Mbar_{1,1}=\G_m\dbs Y$ from Section \ref{stack}, the bundle $\Hbar^\DR_{\vec{1}}\to Y$ descends to a bundle $\Hbar^\DR$ over $\Mbar_{1,1}$ with regular singularity and nilpotent residue at the cusp. It is an extension of $\cH^\DR$ over $\M_{1,1}$ to $\Mbar_{1,1}$.

\subsection{The flat vector bundle $\cH^\an$ and the comparison between $\cH^\an$ and $\cH^\DR$}
Define the holomorphic vector bundle
$$
\cH^\an:=\H^\B\otimes\OO_{\M_{1,1}^\an}
$$
over $\M_{1,1}^\an$. The pullback of $\cH^\an$ to $\h$ (using the quotient map $\h\to\M_{1,1}^\an=\SL_2(\Z)\dbs\h$) is the vector bundle
$$
\cH_\h^\an=\OO_\h\a\oplus\OO_\h\bb,
$$
where the sections $\a$ and $\bb$ are flat.

Define a holomorphic section $\w$ of $\cH_\h^\an$ by 
$$
\w(\tau) = w_\tau = 2\pi i(\adual + \tau \bdual) = 2\pi i(\tau \a - \bb),
$$
where $w_\tau$ is the class in $H^1(E_\tau;\C)$ represented by the holomorphic differential $2\pi i\,d\xi$.

The sections $\a$ and $\w$ trivialize the pull back 
$$\cH^\an_{\D^*}:=\OO_{\D^*}\a\oplus\OO_{\D^*}\w$$ of $\cH^\an$ to $\D^*$ via the map 
$$\h\to\D^*,\qquad\tau\mapsto q:=e^{2\pi i \tau},$$
as they are invariant under $\gamma:\tau\mapsto \tau+1$ (with $\gamma$ being $\pm\begin{pmatrix} 1 & 1 \cr 0 & 1 \end{pmatrix}$), and thus invariant under the monodromy action on the punctured $q$-disk $\D^*$. 

An easy computation \cite[\S 5.2]{umem} shows that the connection on $\cH_{\D^*}^\an$ in terms of this framing is 
$$
\nabla^\an_0=d+\a\frac{\partial}{\partial\w}\frac{dq}{q}.
$$
Since this connection has a regular singularity at the cusp $q=0$, we can extend $\cH_{\D^*}^\an$ to the $q$-disk $\D$ by defining
$$
\cH_{\D}^\an:=\OO_\D\a\oplus\OO_\D\w.
$$
Therefore, we obtain Deligne's canonical extension $\Hbar^\an$ of $\cH^\an$ to $\Mbar_{1,1}^\an$, endowed with a connection $\nabla_0^\an$ that has a regular singularity at the cusp.

To relate Betti and de Rham sections of $\cH^\an$, we pull back the bundle $\cH_{\vec{1}}^\an$ to $\h$ via the map $\h\to\M_{1,\vec{1}}^\an$ in Section \ref{orbi}, and compare it with $\cH_\h^\an$. 
\begin{proposition}\label{HBDR}
There is a natural isomorphism
$$
(\Hbar^\an,\nabla^\an_0)\cong (\Hbar^\DR,\nabla_0)\otimes_{\OO_{\Mbar_{1,1/\Q}}}\OO_{\Mbar_{1,1}^\an}
$$
induced from the pull back. The sections $\T$ and $\Ss$ that correspond to $dx/y$ and $xdx/y$ respectively, after being pulled back, become $\T=\w/2\pi i$, and $\Ss=(\a-2G_2(\tau)\w)/2\pi i$. 
\end{proposition}
\begin{remark}\label{stdual}
Our formulas for $\T$ and $\Ss$ differ from those in Proposition 5.2 of \cite{umem} by a factor of $2\pi i$. The reason is that the cup product of $dx/y$ and $xdx/y$ is $2\pi i$, and we have multiplied their Poincar\'e duals by $(2\pi i)^{-1}$ to obtain a $\Q$-de Rham basis of the first homology \cite[\S 20]{kzb}, such that $\T=\check\Ss$ and $\Ss=-\check\T$. More explanations are provided in Section \ref{Hddq} below.
\end{remark}

\subsection{The fiber of $\H$ at the cusp}\label{Hddq}
To better understand various Betti and de Rham sections of $\cH^\an$, we observe the fiber $H:=H_\ddq$ at the cusp associated to the tangent vector $\ddq$. One can compute the limit mixed Hodge structure on $H$ (computed in \cite[\S 5.4]{umem}), which is isomorphic to $\Q(0)\oplus\Q(-1)$, with Betti realization
$H^\B=\Q\a\oplus\Q\bb$, and de Rham realization
$H^\DR=\Q\a\oplus\Q\w$. Note that on $H$, $-\bb=\adual=\w/2\pi i$ spans $\Q(-1)$ and $\a=\bdual$ spans $\Q(0)$.

One can think of $H$ as the cohomology $H^1(E_\ddq)$. It is better to work with first homology, which is the abelian quotient of the fundamental group. We use Poincar\'e duality to identify $H_1(E)$ with $H^1(E)(1)$. Therefore, we have $H_1(E_\ddq)=H(1)=\Q(1)\oplus\Q(0)$, with Betti realization $\Q\a\oplus\Q\bb$ and de Rham realization $\Q\AA\oplus\Q\T$, where  
$$
\AA:=\a/2\pi i\text{ and } \T:=\w/2\pi i.
$$
Note that on $H(1)$, $\a=2\pi i\,\AA$ spans $\Q(1)$ and $-\bb=\T$ spans $\Q(0)$.

By Proposition \ref{HBDR}, we can write $\Ss$ in terms of this de Rham framing $\AA$, $\T$ of $\H$ (or in fact $\H(1)$)
$$
\Ss=\AA-2G_2(\tau)\T.
$$
We will use these sections $\AA$, $\Ss$ and $\T$ to write down the universal elliptic KZB connection in later sections.

\section{Eisenstein Elliptic Functions and the Jacobi Form $F(\xi,\eta,\tau)$}
\subsection{Eisenstein series}\label{eis}
A modular form\footnote{In this paper, we will only consider modular forms of level one, i.e. those with respect to the entire modular group $\SL_2(\Z)$.} of weight $k$ is a holomorphic function $f(\tau)$ on the upper half plane $\h$ that satisfies
$$f(\gamma\tau)=(c\tau+d)^k f(\tau),\qquad \text{where } \gamma=\begin{pmatrix} a & b \cr c & d\end{pmatrix}\in\SL_2(\Z) \text{ , }\tau\in\h,\text{ and } \gamma\tau=\frac{a\tau+b}{c\tau+d},$$
and it extends to a holomorphic function on the $q$-disk.

Since $f(\tau+1)=f(\tau)$, it has a Fourier expansion of the form
$$f(\tau)=\sum_{n\in\Z} a_nq^n\qquad\text{where }q=e^{2\pi i \tau}.$$
Since a modular form is holomorphic at the cusp, this $q$-series starts with terms of nonnegative $q$ powers. Moreover, a modular form is called a cusp form if the leading coefficient $a_0$ of its $q$-series is $0$.



\begin{example}{\bf Weight $k$ Eisenstein Series $e_k$.}
For integer $k>0$, $\tau\in\h$, define
$$e_k(\tau):=\sum_{\substack{n,m\\ (n,m)\neq (0,0)}}(n\tau+m)^{-k}.$$
Note that $e_k=0$ if $k$ is odd. For $k\ge 4$, the series for $e_k$ is absolutely convergent. For $k=2$, it has to be summed in a certain way, called {\it Eisenstein summation} \cite[Chap. III]{weil}.
\end{example}

\begin{example}{\bf Normalized Weight $k$ Eisenstein Series $G_k$.}
We will normalize the Eisenstein series following Zagier \cite{zagier}. Define $G_k$ to be zero when $k$ is odd. For $k\ge 1$, define $$G_{2k}(\tau):=\frac{1}{2}\frac{(2k-1)!}{(2\pi i)^{2k}}e_{2k}(\tau).$$ It has Fourier expansion
$$G_{2k}(\tau)=-\frac{B_{2k}}{4k}+\sum_{n=1}^{\infty}\sigma_{2k-1}(n)q^n$$
where $B_k$'s are Bernoulli numbers\footnote{One can define Bernoulli numbers $B_n$ by $\frac{x}{e^x-1}=\sum_{n=0}^{\infty}B_n\frac{x^n}{n!}$. Note that $B_0=1$, $B_1=-1/2$ and that $B_{2k+1}=0$ when $k>0$.} and $\sigma_k(n)=\sum_{d|n}d^k$, with
$$G_{2k}|_{q=0}=-\frac{B_{2k}}{4k}=\frac{(2k-1)!}{(2\pi i)^{2k}}\zeta(2k).$$ 

When $k>2$ is even, $G_k$ is holomorphic on the upper half plane, satisfying $G_k(\gamma\tau)=(c\tau+d)^kG_k(\tau)$, and it is holomorphic at the cusp. Therefore, it is a modular form of weight $k$. 
When $k=2$, we have that $G_2$ satisfies (cf. \cite{zagier}) 
$$G_2(\gamma\tau)=(c\tau+d)^2G_2(\tau)+ic(c\tau+d)/4\pi.$$

It is well known that the ring of all modular forms is the polynomial ring $\Q[G_4,G_6]$.
\end{example}

\subsection{Eisenstein elliptic functions}\label{eisell}
For $k>0$, $\tau\in \mathfrak{h}$ and $\xi\in\C$, define Eisenstein elliptic functions \cite{weil} by
$$E_k(\xi, \tau):={\sum_{n,m}}(\xi+n\tau+m)^{-k}.$$
Note that $E_k$ is absolutely convergent for $k\ge 3$. For $k=1,2$, $E_k$ is summed by {\it Eisenstein summation} \cite[Chap. III]{weil}.
We will need the following formula, which is adapted from equation (9) of Chap. IV in \cite{weil}.
\begin{formula}\label{f2}
\begin{equation*} 
2\pi i \frac{\partial E_1}{\partial \tau}=E_3-E_1E_2.
\end{equation*}
\end{formula}

Now we define functions $P_k(\xi,\tau)$ for $k\geq2$
$$P_k(\xi,\tau):=(2\pi i)^{-k}(E_k(\xi,\tau)-e_k(\tau)).$$
Note that up to a scalar, $P_2$ and $P_3$ are the Weierstrass $\wp$-function $\wp_{\tau}(\xi)$ and its derivative respectively.

These $P_k$'s satisfy recurrence relations (cf. \cite{weil}): for $m\geq 3$, $n\geq 3$,
\begin{align*}
P_mP_n-P_{m+n}&=\frac{(-1)^n}{(n-1)!}\sum_{h=1}^{m-2}\frac{2}{h!}G_{n+h}P_{m-h}
\\&+\frac{(-1)^m}{(m-1)!}\sum_{k=1}^{n-2}\frac{2}{k!}G_{m+k}P_{n-k}
\\&+(-1)^m\frac{2(m+n)}{m!n!}G_{m+n}
\end{align*}
The same relation also holds for $m=2, n\geq2$.

By these relations, we know that the algebra generated by Eisenstein elliptic functions is the ring $\K[P_2,P_3]$ with coefficients in $\K:=\Q[G_4,G_6]$. 
\begin{remark}\label{recur}
If variable $\tau$ is fixed, one can use $P_2, P_3$ to embed the elliptic curve $E_\tau$ into a cubic in $\mathbb P^2$ (see Section \ref{alg}), then $P_k$'s are algebraic functions on this elliptic curve. In particular, if the elliptic curve $E_\tau$ is defined over a field $\K$ of characteristic $0$, then there is an embedding with $G_4, G_6\in\K$, so that $G_k\in\K$ for all $k\ge 4$. Therefore, $P_k$'s are polynomials of $P_2,P_3$ with coefficients in $\K$, i.e. $P_k$'s are in the coordinate ring $\OO(E_{\tau/\K})$, which is a $\K$-algebra generated by $P_2,P_3$. 
\end{remark}

\subsection{The Jacobi forms $F(\xi,\eta,\tau)$ and $F^{\zag}(u,v,\tau)$}

There are two different versions of the Jacobi form $F$: one $F(\xi,\eta,\tau)$ used by Levin--Racinet \cite{levin-racinet}, and another $F^{\zag}(u,v,\tau)$ used by Zagier \cite{zagier}. They are related to each other by 
$$F(\xi,\eta,\tau)=2\pi iF^{\zag}(2\pi i\xi,2\pi i\eta,\tau).$$
In this paper, we will use $F^{\zag}$ exclusively. We list some properties of $F^\zag$ here and leave the most relevant ones to the next subsection. These properties all follow from the Theorem in \cite[\S 3]{zagier}.
\begin{proposition}
The function $F^\zag(u,v,\tau)$ has the following properties:
\begin{itemize}
\item[(1)] $F^\zag(u,v,\tau)=F^\zag(v,u,\tau)$.

\item[(2)] $F^\zag(u+2\pi i,v,\tau)=F^\zag(u,v,\tau)$.

\item[(3)] $F^\zag(u+2\pi i\tau,v,\tau)=\exp(-v)F^\zag(u,v,\tau)$.

\item[(4)] $F^\zag(\frac{u}{c\tau+d},\frac{v}{c\tau+d},\gamma\tau)=(c\tau+d)\exp(\frac{cuv}{2\pi i(c\tau+d)})F^\zag(u,v,\tau)$, where $\gamma=\begin{pmatrix} a & b \cr c & d\end{pmatrix}\in\SL_2(\Z)$ and $\gamma\tau=\frac{a\tau+b}{c\tau+d}$.
\end{itemize}
\end{proposition}

\subsection{Some useful formulas}
In this section, we provide some formulas that will be used in later sections.

First, we express the Jacobi form $F^\zag$ in terms of Eisenstein elliptic functions $P_k=(2\pi i)^{-k}(E_k-e_k)$ for $k\ge 2$ and $(2\pi i)^{-1}(E_1-e_1)=(2\pi i)^{-1}E_1$. This is well-known, stated as Proposition-Definition 4.iii) in \cite[\S 3.4]{brown-levin}, also stated as equation (13) in \cite{levin-racinet} but with a typo\footnote{The equation (13) in \cite{levin-racinet} is missing a $\xi$ on the left hand side.}. 
\begin{formula}\label{f1}
\begin{equation*} 
TF^\zag(2\pi i\xi,T,\tau)=\exp\Big(-\sum\limits_{k=1}^{\infty}\frac{(-T)^k}{k}P_k(\xi,\tau)\Big).
\end{equation*}
\end{formula}

\begin{proof}
By \cite{zagier} p456, (viii), 
$$F^{\zag}(u,v,\tau)=\frac{u+v}{uv}\exp\Big(\sum_{k>0}\frac{2}{k!}[u^k+v^k-(u+v)^k]G_k(\tau)\Big).$$
Multiplying by $v$, then take logarithm on both sides, we get
\begin{align*}
\log(vF^{\zag}(u,v,\tau))
  &=\log(1+\frac{v}{u})+\sum_{k>0}\frac{2}{k!}[u^k+v^k-(u+v)^k]G_k(\tau)
\\&=-\sum_{k=1}^{\infty}\frac{(-v/u)^k}{k}+\sum_{k=1}^{\infty}\frac{2v^k}{k!}[1+(\frac uv)^k-(1+\frac uv)^k]G_k(\tau)
\end{align*}
Let $u=2\pi i \xi$, $v=T$, and rescale $G_k$ back to $e_k$, we have
\begin{align*}
\log(TF^\zag(2\pi i\xi,T,\tau))&=-\sum_{k=1}^{\infty}\frac{(-T/2\pi i \xi)^k}{k}+\sum_{k=1}^{\infty}\frac{2T^k}{k!}\left[1+(\frac{2\pi i\xi}{T})^k-(1+\frac{2\pi i\xi}{T})^k\right]\frac12\frac{(k-1)!}{(2\pi i)^k}e_k(\tau)
\\&=-\sum_{k=1}^{\infty}\frac{(-T)^k}{k}(2\pi i\xi)^{-k}-\sum_{l=1}^{\infty}\frac{T^l}{l}(2\pi i)^{-l}\sum_{k=1}^{l-1}{l \choose k}(\frac{2\pi i \xi}{T})^{l-k}e_l(\tau)
\\&=-\sum_{k=1}^{\infty}\frac{(-T)^k}{k}(2\pi i)^{-k}\frac{1}{\xi^k}-\sum_{k=1}^{\infty}\frac{T^k}{k}(2\pi i)^{-k}\sum_{l=k+1}^{\infty}\frac{k}{l}{l \choose k}\xi^{l-k}e_l(\tau)
\\&=-\sum_{k=1}^{\infty}\frac{(-T)^k}{k}(2\pi i)^{-k}\frac{1}{\xi^k}-\sum_{k=1}^{\infty}\frac{(-T)^k}{k}(2\pi i)^{-k}(-1)^k\sum_{l=k+1}^{\infty}{l-1\choose k-1}\xi^{l-k}e_l(\tau)
\\&=-\sum_{k=1}^{\infty}\frac{(-T)^k}{k}(2\pi i)^{-k}\left[\frac{1}{\xi^k}+(-1)^k\sum_{l=k+1}^{\infty}{l-1\choose k-1}\xi^{l-k}e_l(\tau)\right]
\\&=-\sum_{k=1}^{\infty}\frac{(-T)^k}{k}(2\pi i)^{-k}(E_k-e_k)=-\sum_{k=1}^{\infty}\frac{(-T)^k}{k}P_k(\xi,\tau)
\end{align*}
where the last line follows from equation (10) of Chap. III in \cite{weil}, and the facts that ${l-1\choose k-1}=0$ for $l<k$ and that $e_k(\tau)=0$ for odd $k$. After taking exponential on both sides, (\ref{f1}) follows.
\end{proof}

Taking partial derivative with respect to $T$, we have
\begin{formula}\label{cor of f1}
\begin{equation*} 
T\frac{\partial F^\zag}{\partial T}(2\pi i\xi,T,\tau)=\exp\Big(-\sum\limits_{k=1}^{\infty}\frac{(-T)^k}{k}P_k(\xi,\tau)\Big)\Big(\sum\limits_{k=1}^\infty (-T)^{k-1}P_k(\xi,\tau)-\frac{1}{T}\Big).
\end{equation*}
\end{formula}

\section{Unipotent Completion of a Group and its Lie Algebra}\label{unip}

Suppose we are given a finitely generated group $\Gamma$, and a field $\K$ of characteristic 0. The group algebra $\K\Gamma$ is naturally a Hopf algebra with coproduct, antipode and augmentation given by
$$\Delta:g\mapsto g\otimes g,\quad i:g\mapsto g^{-1},\quad \epsilon:g\mapsto 1.$$
Note that it is cocommutative but not necessarily commutative, and thus does not correspond to the coordinate ring of an (pro-)algebraic group. It is natural to consider its (continuous) dual, which is commutative. We define the unipotent completion $\Gamma^{\un}$ of $\Gamma$ over $\K$, an (pro-)algebraic group, by its coordinate ring
$$\OO(\Gamma^{\un}_{/\K})=\Hom_{\cts}(\K\Gamma,\K):=\varinjlim_n\Hom(\K\Gamma/I^n,\K),$$
where we give $\K\Gamma$ a topology by powers of its augmentation ideal $I:=\ker\epsilon$. The set of its $\K$-rational points $\Gamma^\un(\K)$ is in 1-1 correspondence with the set of $$\{\text{ring homomorphisms }\OO(\Gamma^\un_{/\K})\to\K\}.$$
For example, any $\gamma\in\Gamma$ gives a ring homomorphism $\OO(\Gamma^\un_{/\K})\to\K$ by evaluating $\OO(\Gamma^\un_{/\K})$ at $\gamma$, thus determines a $\K$-point $\gamma\in\Gamma^\un(\K)$.

For the purpose of this paper, we only need to consider the case of $\Gamma$ being a free group.

\subsection{The unipotent completion of a free group}\label{unifree}

Suppose that $\Gamma$ is the free group $\langle x_1,\dots,x_n\rangle$ generated by the set $\{x_1,\dots,x_n\}$. The coordinate ring $\OO(\Gamma^\un)$ of its unipotent completion $\Gamma^\un$ over $\K$ is a $\K$ vector space spanned by a basis $\{a_I\}$ indexed by tuples $I=(i_1,i_2,\cdots, i_r)$, where $i_j\in\{1,2,\cdots,n\}$. If the index is empty, then $a_{\varnothing}\equiv 1$; if the index tuple only consists of one number $I=(i)$, we will simply write $a_I$ as $a_i$. The product structure on $\OO(\Gamma^\un)$ is induced by shuffle product $\sha$ and linearity $$a_I\cdot a_J=\sum_{K\in\ I\sha J}a_K.$$
For each $\K$-point $\gamma\in\Gamma^\un(\K)$, ``coordinate function'' $a_I$ takes value $a_I(\gamma)$ in $\K$, and \begin{equation}\label{sh}
a_I(\gamma)\cdot a_J(\gamma)=\sum_{K\in\ I\sha J}a_K(\gamma).
\end{equation}
Note that it is natural to define $\{a_1,\cdots,a_n\}$ as the dual basis of $\{x_1,\cdots,x_n\}$, so that $a_i(x_j)=\delta_{ij}$. 

To determine the structure of $\Gamma^\un$, consider the ring $\K\ll X_1,\dots,X_n\rr$ of formal power series in the non-commuting indeterminants $X_j$. It is a Hopf algebra with each $X_j$ being primitive, and its augmentation ideal is the maximal ideal $I=(X_1,\dots,X_n)$. 

There is a unique group homomorphism
\begin{align*}
\theta:\Gamma &\to \K\ll X_1,\dots,X_n\rr\\
\gamma &\mapsto \sum_I a_I(\gamma)X_I
\end{align*}
that takes $x_j$ to $\exp(X_j)$, where for $I=(i_1,i_2,\cdots, i_r)$, define $X_I:=X_{i_1}X_{i_2}\cdots X_{i_r}$. 

For any $\K$-point $\gamma\in\Gamma^\un(\K)$, the element $\sum_I a_I(\gamma)X_I$ is group-like by (\ref{sh}). This induces a continuous isomorphism
$$
\hat{\theta} : \Gamma^\un(\K) \to \{\text{group-like elements in }\K\ll X_1,\dots,X_n\rr^{\wedge}\},
$$
where $\K\ll X_1,\dots,X_n\rr^{\wedge}$ is completed from $\K\ll X_1,\dots,X_n\rr$ with respect to its augmentation ideal.

It is easy to use universal mapping properties to prove:

\begin{proposition}
The homomorphism $\hat{\theta}$ is an isomorphism of complete Hopf algebras.
\qed
\end{proposition}

\begin{corollary}
The restriction of $\hat{\theta}$ induces a natural isomorphism
$$
d\hat{\theta} : \Lie(\Gamma^\un(\K)) \to \LL_{\K}(X_1,\dots,X_n)^\wedge
$$
of topological Lie algebras.
\end{corollary}

\begin{proof}
This follows immediately from the fact that $\hat{\theta}$ induces an
isomorphism on primitive elements and the well-known fact that the set of
primitive elements of the power series algebra $\K\ll X_1,\dots,X_n\rr$ is the
completed free Lie algebra $\LL_{\K}(X_1,\dots,X_n)^\wedge$.
\end{proof}

\begin{remark}\label{BCH}
By the Baker-Campbell-Hausdorff formula, the exponential map 
$$
\exp: \LL_{\K}(X_1,\dots,X_n)^\wedge\to\{\text{group-like elements in }\K\ll X_1,\dots,X_n\rr^{\wedge}\},
$$
is a group isomorphism. Therefore, $\Gamma^\un(\K)$ and its Lie algebra $\Lie(\Gamma^\un(\K))$ are isomorphic as groups.
\end{remark}

\section{Universal Elliptic KZB Connection---Analytic Formula}\label{KZB}

In this section, we describe the main object to be studied in this paper, working complex analytically. There is a canonical vector bundle $\mathcal P$ (resp. $\p$) over $\mathcal{M}_{1,2}$ whose fiber over a moduli point $[E',x]$ is the unipotent fundamental group $\pi_1^{\un}(E',x)$ (resp. $\Lie(\pi_1^{\un}(E',x))$).\footnote{From Remark \ref{BCH}, the unipotent completion of a group and its Lie algebra are isomorphic, we will regard this bundle as a local system of both unipotent fundamental groups and Lie algebras, whichever is appropriate.} This vector bundle comes with an integrable connection, which is called the universal elliptic KZB connection. Analytic formulas for this connection have been given in different forms by Levin and Racinet \cite{levin-racinet} and by Calaque, Enriquez and Etingof \cite{cee}.

The universal elliptic KZB connection for the bundle $\cP$ over $\E'$ actually lives on $\E$, even $\Ebar$. Since $\cP$ over $\E'$ is a unipotent vector bundle, using Deligne's canonical extension, we obtain $\Pbar$ over $\Ebar$ by extending $\cP$ across the boundary divisors, the identity section and the nodal cubic. The universal elliptic KZB connection has regular singularities around these divisors, as is shown in \cite[\S 12, \S13]{kzb}. 

By Section \ref{unifree}, the fiber of $\cP$ over a point $[E',x]$ is the Lie algebra $\Lie(\pi_1^{\un}(E',x))$ of its unipotent fundamental group $\pi_1^\un(E',x)$, which can be identified with $\LL_\C(\AA,\T)^\wedge$, where $\AA$ and $\T$ are the sections of $\H$ defined in \S\ref{Hddq}. Note that these sections, when pulled back to $\h$, trivialize the vector bundle $\H_\h$, with a factor of automorphy \footnote{The factor of automorphy on $\H_\h$ is easily computed to be
$
M_\gamma(\tau)=
\begin{pmatrix}
 (c\tau+d)^{-1} & 0 \cr
 c & c\tau+d
\end{pmatrix}
$
for $\gamma=\begin{pmatrix} a & b \cr c & d \end{pmatrix}\in\SL_2(\Z)$, cf. \cite[Example 3.4]{kzb}. }. This factor of automorphy lifts to a general one on the bundle $\cP$ over $\C\times\h$\footnote{The general factor of automorphy on $\cP$ is
$
\widetilde{M}_\gamma(\xi,\tau)=
\begin{cases}
M_\gamma(\tau)\circ\exp(\frac{c\xi\T}{2\pi i(c\tau+d)}) & \gamma=\begin{pmatrix} a & b \cr c & d \end{pmatrix}\in\SL_2(\Z) \\
\exp(-m\T) & (m,n)\in\Z^2
\end{cases}
$, cf. \cite[\S 6 (6.2)]{kzb}.}.

We now write the connection form in terms of analytic coordinates $(\xi,\tau)$ on $\C\times\h$. It is shown to be $\SL_2(\Z)\ltimes \Z^2$-invariant in terms of the factor of automorphy we just described and flat in \cite[\S 9]{kzb}. Therefore, it descends to a flat connection on the bundle $\cP$ over the orbifold $\E$.

The connection is defined by $$\nabla^\an f=df+\omega f,$$ with a 1-form 
$$\omega\in\Omega^1((\C\times\h)\log\Lambda_\h)\otimes\Der\,\LL_\C(\AA,\T)^\wedge,$$ whose analytic formula is given by
$$\omega=2\pi i\,d\tau\otimes\AA\frac{\partial}{\partial \T}+\psi+\nu,$$
with
$$\psi=\sum\limits_{m\ge 1}\Big(\frac{G_{2m+2}(\tau)}{(2m)!}2\pi i\,d\tau\otimes\sum\limits_{\substack{j+k=2m+1\\ j,k>0}}(-1)^j[\ad_{\T}^j(\AA),\ad_{\T}^k(\AA)]\frac{\partial}{\partial \AA}\Big),$$
and
$$\nu=\nu_1+\nu_2=\T F^{\zag}(2\pi i\xi,\T,\tau)\cdot \AA\otimes 2\pi i\,d\xi+(\frac{1}{\T}+\T\frac{\partial F^{\zag}}{\partial \T}(2\pi i\xi,\T,\tau))\cdot\AA\otimes 2\pi i\,d\tau.$$
Here, we view $\LL_{\C}(\AA,\T)^\wedge$ as a Lie subalgebra of $\Der\,\LL_{\C}(\AA,\T)^\wedge$ via the adjoint action, and $\T^n$ denotes the $n$-time adjoint action $\ad_\T^n$ on $\LL_{\C}(\AA,\T)^\wedge$; since every derivation $\delta\in\Der\,\LL_{\C}(\AA,\T)^\wedge$ is determined by its values on $\AA$ and $\T$, it can be written uniquely in the form
$$
\delta=\delta(\AA)\frac{\partial}{\partial\AA}+\delta(\T)\frac{\partial}{\partial\T}.
$$

\part{KZB Connection on a Single Elliptic Curve}

In this part, we describe an algebraic de Rham structure $\Pbar_\DR$ on the restriction of the canonical bundle $\Pbar$ to a single elliptic curve $E$. We essentially reproduce and then complete the unfinished work of Levin--Racinet \cite[\S 5]{levin-racinet}. In particular, their connection, though being algebraic, has an irregular singularity at the identity of the elliptic curve. Moreover, their formula is not explicit.

We compute explicitly the restriction of the universal elliptic KZB connection to a single elliptic curve in terms of its algebraic coordinates. We resolve the issue of irregular singularities at the identity in the connection formula by trivializing the bundle $\Pbar_\DR$ on two open subsets of $E$, one of which contains a neighborhood of the identity where the connection has a regular singularity. Therefore, we have constructed a de Rham structure $\Pbar_\DR$ on $\Pbar$ over $E$. Trivializing it on different open subsets is necessary because Deligne's canonical extension $\Pbar$ of $\cP$ from $E'$ (elliptic curve $E$ punctured at the identity) to $E$, unlike the genus zero case (of $\mathbb P^1$), is not trivial as an algebraic vector bundle. 

\section{Elliptic Curves as Algebraic Curves}\label{alg}
Fix $\tau\in\h$ and an elliptic curve $E=E_\tau$. Using the Weierstrass $\wp$-function
$$\wp_{\tau}(\xi):=E_2(\xi,\tau)-e_2(\tau),$$ we can embed a punctured elliptic curve $E'$ into $\mathbb{P}^2$ as follows:
$$\xi\mapsto[(2\pi i)^{-2}\wp_{\tau}(\xi),(2\pi i)^{-3}\wp'_{\tau}(\xi),1].\footnote{We choose this embedding so that powers of $2\pi i$ will not appear in our algebra formulas of KZB connections later.}$$
This satisfies an affine equation
$$y^2=4x^3-ux-v,$$
where $u=g_2(\tau)=20G_4(\tau), v=g_3(\tau)=\frac73 G_6(\tau).$
It is defined over $\K:=\Q(u,v)$. The identity of $E$ is at the infinity. The equation $y=0$ picks out three nontrivial order 2 elements in $E$ (the trivial one being the identity), we define $$E'':=E-\{y=0\}.$$ Note that $\id\in E''$, and $\{E', E''\}$ form an open cover of $E$.

By pulling back through the above embedding, one can identify algebraic functions and forms with their analytic counterparts, which is how we will turn the analytic formula of the connection into an algebraic formula. For example, coordinate functions $x,-\frac y2$ pull back to $P_2, P_3$ defined in Section \ref{eisell}, and the differential $\frac{dx}{y}$ pulls back to $2\pi i\, d\xi$. Note from Remark \ref{recur} that for $k\ge 2$, $P_k$ can be expressed by a polynomial of $P_2,P_3$, i.e. $P_k=P_k(x,y)\in\K[x,y]$.

\section{Algebraic Connection Formula over $E'$}\label{algconn}

Fix $\tau\in\h$, an elliptic curve $E=E_\tau$ defined over a field $\K$ of characteristic zero, and its algebraic embedding as in the last section. The elliptic KZB connection restricted from the universal one to the once punctured elliptic curve $E'=E-\{\id\}$ is
\begin{equation}\label{nu1}
\nabla^\an=d+\nu_1=d+\T F^{\zag}(2\pi i\xi,\T,\tau)\cdot \AA\otimes 2\pi i\,d\xi.
\end{equation}

Note that when described in terms of sections $\AA$ and $\T$, the bundle $\cP$ has factors of automorphy \cite[\S 6]{kzb}. We would like to make sections of $\cP$ elliptic (i.e. periodic with respect to the lattice $\Lambda_\tau$) by a gauge transformation so that the factor of automorphy is absorbed into it and become trivial. The connection form would also be elliptic and can be expressed in terms of algebraic coordinate functions $x$, $y$, and $P_k$'s. Following Levin--Racinet \cite[\S 5]{levin-racinet} and using Formula \ref{f1}, the connection tranforms under the gauge $g_{\alg}(\xi)=\exp(-\frac{1}{2\pi i}E_1\T)$ into $\nabla=d+\nu_1^{\alg}$, with 1-form
\begin{align*}
\nu_1^{\alg}&=-dg_{\alg}\cdot g_{\alg}^{-1}+g_{\alg}\,\nu_1\,g_{\alg}^{-1}
\\&=-\frac{1}{2\pi i}E_2\T\,d\xi+\exp\Big(-\frac{E_1}{2\pi i}\T\Big)\exp\Big(-\sum\limits_{k=1}^{\infty}\frac{(-\T)^k}{k}P_k(\xi,\tau)\Big)\cdot\AA\otimes 2\pi i\,d\xi
\\&=-(2\pi i)^{-2}(E_2-e_2)\T\otimes2\pi id\xi+\exp\Big(-\sum\limits_{k=2}^{\infty}\frac{(-\T)^k}{k}P_k(\xi,\tau)\Big)\cdot(\AA-(2\pi i)^{-2}e_2\T)\otimes 2\pi i\,d\xi
\\&=-\frac{xdx}{y}\T+\exp\Big(-\sum\limits_{k=2}^{\infty}\frac{(-\T)^k}{k}P_k(x,y)\Big)\cdot\Ss\otimes 2\pi i\,d\xi
\\&=-\frac{xdx}{y}\T+\frac{dx}{y}\Ss+\sum_{n=2}^{\infty}q_n(x,y)\frac{dx}{y}\T^n\cdot \Ss.
\end{align*}
Here $\nu_1^{\alg}\in\Omega^1(E'_{/\K})\otimes\LL_\K(\Ss,\T)^\wedge$, $\T^n$ is the action $\ad_\T^n$ on $\LL_\K(\Ss,\T)^\wedge$ that repeats the adjoint action of $\T$ for $n$ times, and $$q_n(x,y)=\sum_{2a_2+3a_3+\cdots+na_n=n}\frac{1}{a_2!a_3!\cdots a_n!}\prod_{k=2}^n\left(\frac{(-1)^{k+1}P_k(x,y)}{k}\right)^{a_k}\in\OO(E'_{/\K}),$$
where $\OO(E'_{/\K})=\K[x,y]/(y^2-4x^3+ux+v).$ Note that the above sum is indexed by the partitions of integer $n$ with summands at least 2. For example, $5$ has 2 such partitions: $5$ and $2+3$, so $q_5=\frac{P_5}{5}+(-\frac{P_2}{2})\cdot\frac{P_3}{3}=\frac15P_5-\frac16P_2P_3$. Written as polynomials in $\K[x,y]$, we have $q_2=-\frac{P_2}{2}=-\frac12x$, $q_3=\frac{P_3}{3}=-\frac16y$, and $q_4=(-\frac{P_4}{4})+\frac12(-\frac{P_2}{2})^2=\frac{u}{40}-\frac{x^2}{8}$.
\begin{remark}
One can use the recurrence relations of $P_k$'s described in Section \ref{eisell} to find relations among $q_n$'s.
\end{remark}

Note that the form $\nu_1^{\alg}$ is defined over $\K$, so we have constructed an algebraic vector bundle $(\cP_\DR,\nabla)$ over $E'$ whose fibers can be identified with $\LL_\K(\Ss,\T)^\wedge$. This algebraic bundle is defined over $\K$, with its connection $\nabla$ also defined over $\K$. It provides us with a $\K$-structure $\cP_\DR$ on $\cP$ over $E'$. Since the form $\nu_1^{\alg}$ has irregular singularity ($\frac{xdx}{y}$ having a double pole) at the identity, we cannot extend it naively across the identity to obtain Deligne's canonical extension. To construct the canonical extension, we have to change gauge on a Zariski open neighborhood $E''$ of the identity. We do this in Section \ref{regconn}.

\subsection{The naive connection vs. the elliptic KZB connection}\label{naive}
Before we do this, we consider the naive connection on the trivial bundle $$\LL_\K(\Ss,\T)^\wedge\times E'\to E'$$ which is defined by $\nabla'=d+\nu_1^{\naive}$, where $$\nu_1^\naive = -\frac{xdx}{y}\T+\frac{dx}{y}\Ss.\footnote{The ``$-$"-sign appears as we regard $-\T$ and $\Ss$ as a basis for $H_1(E)$, dual to $\Ss=xdx/y$ and $\T=dx/y$ in $H^1(E)$, see Remark \ref{stdual} before Section \ref{Hddq}.}$$
This flat connection is defined over $\K$, whose monodromy is a homomorphism $\rho^\naive:\pi_1(E',x)\to\LL(\Ss,\T)^\wedge$ that induces an isomorphism
$$I^\naive:\Lie\,\pi_1^\un(E',x)\to\LL(\Ss,\T)^\wedge.$$
 Since the elliptic KZB connection $\nabla=d+\nu_1^{\alg}$ agrees with the naive connection up to degree 2, its monodromy $\rho^\KZB:\pi_1(E',x)\to\LL(\Ss,\T)^\wedge$ also induces an isomorphism 
$$I^\KZB:\Lie\,\pi_1^\un(E',x)\to\LL(\Ss,\T)^\wedge.$$ 
Although both $I^\naive$ and $I^\KZB$ are isomorphisms of prounipotent groups, they can not be algebraically transfered from one to the other while preserving the group structure on $\LL(\Ss,\T)^\wedge$.

\begin{proposition}\label{Kstr}
Fix a field $\K$ of characteristic zero. Between the elliptic KZB connection $\nabla=d+\nu_1^{\alg}$ and the naive connection $\nabla'=d+\nu_1^{\naive}$, there are no algebraic gauge transformation over any Zariski open subset of $E_{/\K}$ that preserves the group structure on the fiber $\LL(\Ss,\T)^\wedge$.
\end{proposition}
\begin{proof}
Suppose there were such a change of gauge $g$ that preserves the group structure, its value would lie in the subgroup $\exp\LL(\Ss,\T)^\wedge$ of $\Aut\LL(\Ss,\T)^\wedge$, acting on the fiber $\LL(\Ss,\T)^\wedge$ by conjugation. In other words, we would have $g: E\DashedArrow[->,densely dashed] \exp\LL(\Ss,\T)^\wedge$, with coefficients in $\K(E)$, field of fractions of $\OO(E'_{/\K})$, such that
\begin{equation*}
\nu_1^{\alg}=-dg\cdot g^{-1}+g\nu_1^{\naive}g^{-1},
\end{equation*}
or equivalently
\begin{equation}\label{gauge}
dg=g\nu_1^{\naive}-\nu_1^{\alg}g.
\end{equation}
This is an equation of 1-forms on $E$ with values in $\Der\,\LL(\Ss,\T)^\wedge$. Now let
\begin{align*}
g=1+\alpha\T+&\beta\Ss+\gamma\T^2+\lambda\Ss\T+\mu\T\Ss+\delta\Ss^2
\\&+\sigma\T^3+\zeta\T^2\Ss+\eta\T\Ss\T+\xi\T\Ss^2+\tau\Ss\T^2+\kappa\Ss\T\Ss+\epsilon\Ss^2\T+\iota\Ss^3+\cdots
\end{align*}
where coefficients $\alpha,\beta,\gamma,\cdots\in\K(E)$ should be regarded as rational functions on the elliptic curve $E_{/\K}$.
We substitute $g$ into the above equation (\ref{gauge}). Now we equate the coefficients up to the third degree of derivations in $\Der\,\LL(\Ss,\T)^\wedge$. We have
\begin{align}
&d\alpha=0,\quad d\beta=0,\quad d\gamma=0,\quad d\delta=0,\quad d\sigma=0\label{S} \\
&d\mu=\alpha\,\frac{dx}{y}+\beta\,\frac{x\,dx}{y}\label{TS}\\
&d\lambda=-\beta\,\frac{x\,dx}{y}-\alpha\,\frac{dx}{y}\label{ST}\\
&d\zeta=\gamma\,\frac{dx}{y}+\mu\,\frac{x\,dx}{y}+\frac{1}{2}\frac{x\,dx}{y}\label{TTS}\\
&d\eta=-\mu\,\frac{x\,dx}{y}+\lambda\,\frac{x\,dx}{y}\label{TST}\\
&d\xi=\mu\,\frac{dx}{y}+\delta\frac{x\,dx}{y}\label{TSS}\\
&\cdots
\end{align}
From (\ref{S}), we know that $\alpha$ and $\beta$ are constants. Taking cohomology classes on both sides of (\ref{TS}), we have $\alpha\,[\frac{dx}{y}]+\beta\,[\frac{x\,dx}{y}]=0$, and easily get $\alpha=\beta=0$. Thus $d\mu=0$, and $\mu$ is a constant. For the same reason, $\lambda$ is a constant.

Now taking cohomology classes on both sides of (\ref{TTS}) and of (\ref{TST}), we get $\gamma=0$, and $\lambda=\mu=-\frac12$. Similarly, taking cohomology classes on both sides of (\ref{TSS}), we get $\mu=\delta=0$. However, $\mu$ cannot be $-\frac12$ and $0$ at the same time!
\end{proof}

\begin{remark}\label{age}
If we forget about the group structure on the fiber $\LL(\Ss,\T)^\wedge$, we expect that there is an algebraic gauge transformation, defined over $\K$ and meromorphic at the identity when working over $\C$, between the elliptic KZB connection and the naive connection. The reason is that one expects to be able to solve the above equation (\ref{gauge}), if the gauge transformation $g$ is allowed to take value in $\Aut\LL(\Ss,\T)^\wedge$. This would indicate that periods of (regularized) iterated integrals constructed from both connections are the same. For some evidence of this, up to degree 5, one can take $g$ to be 
\begin{align*}
g: \Ss &\mapsto \Ss-\frac{29}{960}u[\T,[\T,[\T,[\T,\Ss]]]]-\frac16 x[\Ss,[\T,[\T,[\T,\Ss]]]] + \cdots\\
\T &\mapsto \T+\frac12[\T,[\T,\Ss]]-\frac16 x[\T,[\T,[\T,[\T,\Ss]]]]-\frac16[\Ss,[\T,[\T,[\T,\Ss]]]] + \cdots
\end{align*}
The form of this seems to suggest that there are (cohomological) obstructions in degrees $3,5,\cdots$ to gauge transform between the elliptic KZB connection and the naive connection while preserving the group structure on $\LL(\Ss,\T)^\wedge$. %
\end{remark}

\section{Algebraic Connection Formula over $E''$}\label{regconn}
Recall that we have an analytic local system $(\cP,\nabla^\an)$ of (Lie algebras of) unipotent fundamental groups over $E'$. The elliptic KZB connection $\nabla^\an$ is obtained by restricting the universal elliptic KZB connection to $E'$. Note that the analytic formula of the elliptic KZB connection $\nabla^\an$ has regular singularity at the identity with pronilpotent residue. The elliptic KZB connection thus extends naturally from $E'$ to $E$, and we obtain Deligne's canonical extension $(\Pbar,\nabla^\an)$ of $(\cP,\nabla^\an)$. It is not immediately clear that $(\Pbar,\nabla^\an)$ has an algebraic de Rham structure. The question is to determine whether this canonical extension is defined over $\K$, the field of definition of $E$. In this section, we show that the elliptic KZB connection $\nabla^\an$ is gauge equivalent to its algebraic counterpart $\nabla$ defined over $\K$, which has regular singularity at the identity with pronilpotent residue. It follows that Deligne's canonical extension $\Pbar$ of $\cP$ to $E$ is defined over $\K$.

We start with the algebraic connection $\nabla=d+\nu_1^\alg$, which is defined to be gauge equivalent to the analytic one $\nabla^\an$ in the last section. Since $\nu_1^\alg$ has irregular singularities at the identity of $E$,  we would like to apply another gauge transformation to make it regular. The reason $\nu_1^\alg$ has irregular singularity is that we introduced a gauge transformation involving $E_1$, which has a pole at the identity. To cancel this effect and make the connection regular at the identity, we apply a gauge transformation $g_{\reg}=\exp(-\frac{2x^2}{y}\T)$. Then the connection becomes $\nabla=d+\nu_1^{\reg}$, with 1-form
\begin{align*}
\nu_1^{\reg}&=-dg_{\reg}\cdot g_{\reg}^{-1}+g_{\reg}\,\nu_1^{\alg}\,g_{\reg}^{-1}
\\&=\left(d\left(\frac{2x^2}{y}\right)-\frac{xdx}{y}\right)\T+\exp\Big(-\frac{2x^2}{y}\T-\sum\limits_{k=2}^{\infty}\frac{(-\T)^k}{k}P_k(x,y)\Big)\cdot\Ss\otimes 2\pi i\,d\xi
\\&=\left(d\left(\frac{2x^2}{y}\right)-\frac{xdx}{y}\right)\T+\frac{dx}{y}\Ss+\sum_{n=1}^{\infty}r_n(x,y)\frac{dx}{y}\T^n\cdot \Ss.
\end{align*}
Here $\nu_1^{\reg}\in\Omega^1(E''\log\{\id\})\otimes\LL_\K(\Ss,\T)^\wedge$ and we get, by direct calculation, rational functions $$r_n(x,y)=\sum_{a_1+2a_2+3a_3+\cdots+na_n=n}\frac{1}{a_1!a_2!a_3!\cdots a_n!}\prod_{k=1}^n\left(\frac{(-1)^{k+1}P_k(x,y)}{k}\right)^{a_k}\in\OO(E''-\{\id\}),$$
where $P_1(x,y):=-\frac{2x^2}{y}$ and $\OO(E''-\{\id\})=\OO(E'_y)=\OO(E')[y^{-1}]$. Note that the sum for $r_n$ is indexed by the partitions of integer $n$ with no restrictions of the summands. For example, $4$ has 5 partitions: $4$, $3+1$, $2+2$, $2+1+1$, $1+1+1+1$, so 
\begin{align*}
r_4&=(-\frac{P_4}{4})+\frac{1}{1!1!}(\frac{P_3}{3})\cdot(\frac{P_1}{1})+\frac{1}{2!}(-\frac{P_2}{2})^2+\frac{1}{2!1!}(\frac{P_1}{1})^2\cdot(-\frac{P_2}{2})+\frac{1}{4!}(\frac{P_1}{1})^4\\
&=-\frac14P_4+\frac13P_3P_1+\frac18P_2^2-\frac14P_2P_1^2+\frac{1}{24}P_1^4
\end{align*}
\begin{remark}
One can use the recurrence relations of $P_k$'s described in Section \ref{eisell} to find relations among $r_n$'s.
\end{remark}

In the next section, we will check that
\begin{lemma}\label{reg}
The connection $\nabla=d+\nu_1^{\reg}$ has a regular singularity at the identity with pronilpotent residue. 
\end{lemma}
Therefore, this connection a priori living on $E''-\{\id\}$, can be extended naturally across the identity. It is an algebraic connection defined over $\K$ on an algebraic vector bundle $\Pbar_\DR|_{E''}$ over the open subset $E''$ of $E$. This is one part of a vector bundle $\Pbar_\DR$ over $E$. The other part $\Pbar_\DR|_{E'}=\cP_\DR$ was constructed using the connection $\nabla=d+\nu_1^\alg$ in the last section. Now we have trivialized $\Pbar_\DR$ on an open cover of two different subsets of $E$. By gluing two trivializations together in terms of the gauge transformation, we have constructed an algebraic vector bundle $\Pbar_\DR$ over $E$.

Summarizing results in this part, we get

\begin{theorem}[\bf The algebraic de Rham structure $\Pbar_\DR$ on $\Pbar$ over $E$] \label{adsse}
Suppose that $\K$ is a field of characteristic 0, embeddable in $\C$. Let $E$ be an elliptic curve defined over $\K$. Then for each embedding $\sigma:\K\xhookrightarrow{} \C$, we have an algebraic vector bundle $\Pbar_\DR$ over $E_{/\K}$ endowed with connection $\nabla$, and an isomorphism $$(\Pbar_\DR,\nabla)\otimes_{\K,\sigma}\C\approx(\Pbar,\nabla^\an).$$
The algebraic bundle $\Pbar_\DR$ and its connection $\nabla$ are both defined over $\K$. The $\K$-de Rham structure $(\Pbar_\DR,\nabla)$ on $(\Pbar,\nabla^\an)$ is explicitly given by the connection formulas for $\nu_1^\alg$ on $E'$ and $\nu_1^{\reg}$ on $E''$ above. In particular, the connection $\nabla$ has a regular singularity at the identity.
\end{theorem}

\section{Regular Singularity and Residue at the Identity}
In this section, we prove Lemma \ref{reg} by showing that $\nu_1^{\reg}$ has regular singularity at the identity, and we compute its residue there. 

It is easy to check that analytically $d\left(\frac{2x^2}{y}\right)-\frac{xdx}{y}$ is holomorphic at the identity. So we are left to check that
\begin{equation} \label{r1}
1+\sum_{n=1}^{\infty}r_n(x,y)\T^n=\exp\Big(-\sum\limits_{k=1}^{\infty}\frac{(-\T)^k}{k}P_k(x,y)\Big)
\end{equation}
has a regular singularity at the identity.

Let $\xi$ be the complex coordinate near the identity. Analytically, we need to calculate (in terms of $\xi$) the principal parts of $P_1=-\frac{2x^2}{y}$ and $P_k$'s $(k\geq 2)$. The principal part of $P_1=-\frac{2x^2}{y}$ is $\frac{1}{2\pi i \xi}$; the principal part of $P_k$ $(k\geq 2)$ is $\frac{1}{(2\pi i \xi)^k}$, since 
\begin{align*}
P_k &=(2\pi i)^{-k}(E_k-e_k)
\\&=(2\pi i)^{-k}\sum_{m,n}(\xi+m\tau+n)^{-k}-{\sum_{m,n}}'(m\tau+n)^{-k}
\\&=\frac{1}{(2\pi i)^{k}\xi^k}+(2\pi i)^{-k}{\sum_{m,n}}'\left(\frac{1}{(\xi+m\tau+n)^k}-\frac{1}{(m\tau+n)^k}\right)
\\&=\frac{1}{(2\pi i)^{k}\xi^k}+(2\pi i)^{-k}{\sum_{m,n}}'\frac{1}{(m\tau+n)^k}\sum_{l=1}^{\infty}(-1)^l\binom{l+k-1}{k-1}\frac{\xi^l}{(m\tau+n)^l}
\\&=\frac{1}{(2\pi i \xi)^k}+\sum_{l=1}^{\infty}(-1)^l\binom{l+k-1}{k-1}(2\pi i)^{-(k+l)}e_{k+l}\,(2\pi i \xi)^l.
\end{align*}
Therefore, (\ref{r1}) is of the following form near $\xi=0$,
\begin{align*}
\exp\Big(-\sum_{k=1}^{\infty}\frac{(-\T/(2\pi i \xi))^k}{k}+\text{holo. in $\xi$}\Big)
&=\exp\left(\ln(1+\T/(2\pi i \xi))\right)\exp\Big(\sum_{n=0}^{\infty}a_n(\T)(2\pi i \xi)^n\Big)
\\&=(1+\T/(2\pi i \xi))\exp\Big(\sum_{n=1}^{\infty}a_n(\T)(2\pi i \xi)^n\Big),
\end{align*}
which has a regular singularity at the identity. Here $\forall n\geq 0, a_n(\T)\in\K[\T]$ and $a_0(\T)=0$.

Now it's easy to calculate the residue. Note that $\frac{2x^2}{y}$ is an odd function in $\xi$, and when expressed in terms of $\xi$, it has constant term $0$ in the holomorphic part; so does each of the $P_k$'s according to their expansions above. Therefore, we know that the holomorphic part in $\xi$ also has constant term $0$, and the residue at the identity we are looking for is then
$$\frac{\T}{2\pi i}\exp(0)\cdot \Ss(2\pi i)=\T\cdot \Ss=\ad_{[\T,\Ss]},$$
which is in $\Der\,\LL(\Ss,\T)^\wedge$. Note that $(2\pi i)$ at the end of the first expression above comes from $dx/y=2\pi i d\xi$.

\section{Tannaka Theory and a Universal Unipotent Connection over $E$}

Recall that a unipotent object in a tensor category $\Ca$ with the identity object $\bone_\Ca$ is an object $V$ with a filtration in $\Ca$
$$0=V_0\subseteq\cdots\subseteq V_n=V$$
such that each quotient $V_j/V_{j-1}$ is isomorphic to $\bone_\Ca^{\oplus k_j}$ for some $k_j\in\mathbb N$. 

Let $E$ be an elliptic curve defined over $\K$, and fix an embedding $\sigma:\K\xhookrightarrow{}\C$. Let $E'=E-\{\id\}$. Consider the following tensor categories: 
\begin{enumerate}
\item 
{\bf Unipotent Local Systems}
$$\Ca^\B_F:=\{\text{unipotent local systems $\V_F$ over $E'(\C)$}\},$$ 
where $F$ is a field of characteristic 0, and the identity object $\bone_{\Ca^\B_F}$ is the constant sheaf $F_{E'}$ on $E'(\C)$;
\item
{\bf Algebraic de Rham}
$$\Ca^\DR_\K:=\left\{
\begin{tabular}{c}
unipotent vector bundles $\Vbar$ over $E_{/\K}$ defined over $\K$\\
with a flat connection $\nabla$ that has regular singularity\\
at the identity with nilpotent residue
\end{tabular}\right\},$$
where the identiy object $\bone_{\Ca^\DR_\K}$ is the trivial vector bundle $\OO_{E}$ with the trivial connection given by the exterior differential $d$;
\item
{\bf Analytic de Rham}
$$
\Ca^\an:=\left\{
\begin{tabular}{c}
unipotent vector bundles $\Vbar^\an$ over $E^\an$ with a flat connection\\
that is holomorphic over $E'(\C)$, meromorphic over $E(\C)$ and\\
has regular singularity at the identity with nilpotent residue
\end{tabular}\right\},
$$
where $E^\an=E(\C)$ is the analytic variety associated to $E_{/\K}$, and the identiy object $\bone_{\Ca^\an}$ is the trivial vector bundle $\OO_{E^\an}$ with the trivial connection given by the exterior differential $d$.
\end{enumerate}

One can define fiber functors for these tensor categories so that they become neutral tannakian categories. Taking the fiber over $x\in E'(\C)$ of any object in $\Ca^\B_F$ provides a fiber functor $\omega_x$ of $\Ca^\B_F$. By Tannaka duality and the universal property of unipotent completion, the tannakian fundamental group of $\Ca^\B_F$ with respect to the fiber functor $\omega_x$, which we denote by $\pi_1(\Ca^\B_F,\omega_x)$, is the unipotent fundamental group $\pi_1^\un(E',x)_F$ over $F$. We will denote $\pi_1^\un(E',x)_\Q$ simply by $\pi_1^\un(E',x)$.

In the same way, one can define a fiber functor $\omega_x$ of $\Ca^\an$ for any $x\in E(\C)$, and a fiber functor $\omega_x$ of $\Ca^\DR_\K$ for any $x\in E(\K)$. Note that we can take $x$ to be the identity $\id\in E(\K)$. We denote their corresponding tannakian fundamental groups by $\pi_1(\Ca^\an,\omega_x)$ and $\pi_1(\Ca^\DR_\K,\omega_x)$ respectively. Our objective is to establish a natural comparison isomorphism between $\pi_1(\Ca^\an,\omega_x)$ and $\pi_1(\Ca^\DR_\K,\omega_x)\times_\K\C$ for any $x\in E(\K)$.

\subsection{Extension groups in $\Ca^\B_F$, $\Ca^\an$ and $\Ca^\DR_\K$}\label{exts}
We start with a general setting. Let $K$ be a field of characterisitic zero. Let $\Ca$ be a neutral tannakian category over $K$ with a fiber functor $\omega$ all of whose objects are unipotent. Denote its identity object by $\bone_\Ca$. The tannakian fundamental group of $\Ca$ with respect to $\omega$, which we denote by $\U$, is a prounipotent group defined over $K$. Denote its Lie algebra by $\u$, viewed as a topological Lie algebra. Since the category of $\U$-modules is equivalent to the category of continuous $\u$-modules, we have
$$H^m_\cts(\u)\cong H^m(\U)\cong \Ext^m_\Ca(\bone_\Ca,\bone_\Ca).$$ 
The following is standard.
\begin{proposition}
Let $\u$ be a pronilpotent Lie algebra, and denote its abelianization by $H_1(\u)$. Then
$$H_1(\u)\cong\Hom(H^1_\cts(\u),K).$$
If $H^2(\u)=0$, then $\u$ is a free Lie algebra.
\end{proposition}
Therefore, if $\Ext^2_\Ca(\bone_\Ca,\bone_\Ca)=0$, then the Lie algebra $\u$ of the tannakian fundamental group of $\Ca$ is freely generated by $\Ext^1_\Ca(\bone_\Ca,\bone_\Ca)^*$, the $K$-dual of $\Ext^1_\Ca(\bone_\Ca,\bone_\Ca)$.

Now we compute extension groups in categories $\Ca^\B_F$, $\Ca^\an$ and $\Ca^\DR_\K$.
\begin{lemma}
$$\Ext^1_\Ca(\bone_\Ca,\bone_\Ca)\cong
\begin{cases}
H^1(E(\C); F) & \quad \text{when $\Ca=\Ca^\B_F$,} \\
H^1(E(\C); \C) & \quad \text{when $\Ca=\Ca^\an$,}\\
H^1_\DR(E_{/\K}) & \quad \text{when $\Ca=\Ca^\DR_\K$.}
\end{cases}$$
\end{lemma}
\begin{proof}
The first two cases are well known. The third case can be easily obtained by tensoring $\Ca^\DR_\K$ with $\C$ via the fixed embedding $\sigma:\K\xhookrightarrow{}\C$ and invoking Grothendieck's algebraic de Rham theorem, which provides an isomorphism
$$H^1_\DR(E_{/\K})\otimes_\K \C\xrightarrow{\sim}H^1(E(\C);\C);$$
instead, we provide another proof, explicitly constructing extensions by using our algebraic connection formulas on $\Pbar_\DR$. Given a global 1-form $\omega$ on $E$, we can define a connection
$$\nabla=d+\begin{pmatrix} 0 & \omega \cr 0 & 0 \end{pmatrix},$$
on the trivial bundle $\OO_E\oplus\OO_E$. This defines an extension in $\Ext^1_{\Ca^\DR_\K}(\bone_{\Ca^\DR_\K},\bone_{\Ca^\DR_\K})$ and gives rise to a map 
$$e:H^0(E,\Omega^1_E)\to\Ext^1_{\Ca^\DR_\K}(\bone_{\Ca^\DR_\K},\bone_{\Ca^\DR_\K}),$$
which is injective. To see this, we first tensor with $\C$ on both sides of this map. One can then identify the complexified extension group on the right with $H^1(E;\C)$ by using monodromy. And the map becomes the inclusion of holomorphic 1-forms on $E$ into $H^1(E;\C)$, which is injective.

Suppose we have a vector bundle $(\cV,\nabla)\in\Ca^\DR_\K$, which is an extension of $(\OO_E,d)$ by $(\OO_E,d)$. By forgeting the connections on all these bundles, this extension determines a class in $\Ext^1_E(\OO_E,\OO_E)\cong H^1(E,\OO_E)$. This gives rise to a map $f$ and the following sequence
$$0\to H^0(E,\Omega^1_E)\xrightarrow{e} \Ext^1_{\Ca^\DR_\K}(\bone_{\Ca^\DR_\K},\bone_{\Ca^\DR_\K})\xrightarrow{f} H^1(E,\OO_E)\to 0.$$
The result follows if this is a short exact sequence.

Suppose a vector bundle $(\cV,\nabla)$ represents a class in $\ker f$, then we have a split extension (without connection)
$$0\to\OO_E\to\cV\to\OO_E\to 0.$$
Fixing a splitting on $\cV$, the connection can be written as 
$$\nabla=d+\begin{pmatrix} 0 & \omega \cr 0 & 0 \end{pmatrix},$$
where $\omega$ is a global 1-form on $E$. So we have
$$\ker f=\im e.$$

To show $f$ is surjective, we provide here explicitly a vector bundle with connection that corresponds to a nontrivial extension class in $H^1(E,\OO_E)$. Recall that the connection $\nabla$ on $\Pbar_\DR$ is given by algebraic connection formulas
$$\nabla=
\begin{cases}
d+\nu_1^\alg=d-\frac{xdx}{y}\T+\frac{dx}{y}\Ss+\cdots & \quad\text{on $E'$},\\
d+\nu_1^\reg=d+\left(d\left(\frac{2x^2}{y}\right)-\frac{xdx}{y}\right)\T+\frac{dx}{y}\Ss+\cdots & \quad\text{on $E''$}.
\end{cases}
$$
The leading terms recorded here provides a connection on the abelianization 
of $\Pbar_\DR$. This gives a nontrivial extension of $\OO_E$ by $\OO_E$, thus corresponds to a nontrivial class in $H^1(E,\OO_E)$.
\end{proof}

\subsection{The de Rham tannakian fundamental group $\pi_1(\Ca^\DR_\K,\omega_x)$}
It is well known that there is an equivalence of categories
$$\Ca^\B_\C\rightleftarrows\Ca^\an.$$
The right arrow is the functor that takes a unipotent local system $\V$ over $E'$ to Deligne's canonical extension $(\Vbar,\nabla)$ of $\V\otimes\OO^\an_{E'}$, where $$\nabla:\Vbar\rightarrow\Vbar\otimes\Omega^1_E(\log\{\id\}).$$ The left arrow is the functor obtained by taking locally flat sections of $\cV$ over $E'$. By this equivalence, we obtain an isomorphism between their tannakian fundamental groups
\begin{equation}\label{bfg}
\comp_{\an,\B}:\pi_1(\Ca^\an,\omega_x)\xrightarrow{\cong}\pi_1(\Ca^\B_\C,\omega_x)=\pi_1^\un(E',x)\times_\Q\C
\end{equation}
for each $x\in E'(\C)$. By Section \ref{unifree}, as an abstract group, the unipotent fundamental group $\pi_1^\un(E',x)_\C$ can be identified with its Lie algebra $\LL_\C(\AA,\T)^\wedge$, which is the same as $\LL_\C(\Ss,\T)^\wedge$, where $\AA$, $\Ss$ and $\T$ are the sections defined in \S\ref{Hddq}.

The local system $\cP$ over $E'$ is a pro-object in $\Ca^\B_\C$, which is equivalent to an action
\begin{equation}\label{baction}
\pi_1(\Ca^\B_\C,\omega_x)\to\Aut \LL_\C(\AA,\T)^\wedge
\end{equation}
of the tannakian fundamental group on the fiber of $\cP$ over $x$.
This corresponds to the adjoint action of $\LL_\C(\AA,\T)^\wedge$ on itself
\begin{equation}\label{adac}
\ad:\LL_\C(\AA,\T)^\wedge\to\Der\ \LL_\C(\AA,\T)^\wedge.
\end{equation}

There is another equivalence of categories
$$\Ca^\DR_\C\rightleftarrows\Ca^\an,$$
where the right arrow is the obvious one, and the left arrow exists by GAGA: since $E^\an=E(\C)$ is projective, the category of analytic sheaves over $E^\an$ is equivalent to its algebraic counterpart over $\C$. By this equivalence, we have an isomorphism of tannakian fundamental groups
$$\pi_1(\Ca^\an,\omega_x)\to\pi_1(\Ca^\DR_\C,\omega_x)$$
for any $x\in E(\C)$. 
Although it is well known that for each $x\in E(\K)$ one can get a canonical $\K$-structure $\pi_1(\Ca^\DR_\K,\omega_x)$ on $\pi_1(\Ca^\DR_\C,\omega_x)$, we provide an elaborate proof to set up the discussion of universal connection in the next subsection.
\begin{proposition}
There is a natural comparison isomorphism
$$\comp_{\an,\DR}:\pi_1(\Ca^\an,\omega_x)\to\pi_1(\Ca^\DR_\K,\omega_x)\times_\K\C$$
for any $x\in E(\K)$. 
\end{proposition}
\begin{proof}
This map $\comp_{\an,\DR}$ is induced from the functor of tensoring with $\C$ by using the fixed embedding $\sigma:\K\xhookrightarrow{}\C$:
$$\Ca^\DR_\K\otimes\C\to\Ca^\DR_\C\simeq\Ca^\an.$$
We study it by working with a special object. In Section \ref{regconn}, we constructed such an object: an algebraic vector bundle $(\Pbar_\DR,\nabla)$ over $E$ with a connection $\nabla$ defined over $\K$. It is a pro-object in $\Ca^\DR_\K$, and corresponds to an action
\begin{equation}\label{draction}
\pi_1(\Ca^\DR_\K,\omega_x)\to\Aut \LL_\K(\Ss,\T)^\wedge
\end{equation}
of the tannakian fundamental group on the fiber over $x$. Recall from Theorem \ref{adsse} that
$$(\Pbar_\DR,\nabla)\otimes_{\K}\C\approx(\Pbar,\nabla^\an),$$
where $\Pbar$ is Deligne's canonical extension of $\cP$ over $E'$ to $E$. Therefore, after tensoring with $\C$, we obtain the object $\cP$ in $\Ca^\an$, which, by (\ref{bfg}) and (\ref{baction}), is equivalent to an action
$$\pi_1(\Ca^\an,\omega_x)\to\Aut \LL_\C(\Ss,\T)^\wedge.$$
This action factors through the action given by $\text{(\ref{draction})}\times_\K\C$, that is, we have a diagram
$$
\begin{tikzcd}[column sep=huge]
\pi_1(\Ca^\an,\omega_x)\arrow{r}{\comp_{\an,\DR}}\arrow{rd} & \pi_1(\Ca^\DR_\K,\omega_x)\times_\K\C \arrow{d}\\ 
& \Aut\ \LL_\C(\Ss,\T)^\wedge
\end{tikzcd}\\
$$
Note that the functor that takes a unipotent group to its Lie algebra is an equivalence of categories between the category of unipotent $\K$-groups and the category of nilpotent Lie algebras over $\K$. The above diagram is thus equivalent to the following diagram
$$
\begin{tikzcd}
\LL_\C(\AA,\T)^\wedge\arrow[r]\arrow[rd,hook,"\ad"] & \u(\Ca^\DR_\K,\omega_x)\otimes_\K\C \arrow{d}\\ 
& \Der\ \LL_\C(\AA,\T)^\wedge
\end{tikzcd}\\
$$
where $\u(\Ca^\DR_\K,\omega_x)$ denotes the Lie algebra of $\pi_1(\Ca^\DR_\K,\omega_x)$. Since the adjoint action $\ad:\LL_\C(\AA,\T)^\wedge\to\Der\ \LL_\C(\AA,\T)^\wedge$ from (\ref{adac}) is injective, the top row of the previous diagram
$$\comp_{\an,\DR}:\pi_1(\Ca^\an,\omega_x)\to\pi_1(\Ca^\DR_\K,\omega_x)\times_\K\C$$
must also be injective. The surjectivity of this map follows from the fact that the Lie algebra $\u(\Ca^\DR_\K,\omega_x)$ is generated by the $\K$-dual $H^1_\DR(E_{/\K})^*$ of $H^1_\DR(E_{/\K})$, see discussion in Section \ref{exts}.
\end{proof}
\begin{remark}\label{ff}
One can choose $x$ to be a tangential base point. For example, one can take the fiber functor to be the fiber at the unit tangent vector $\vv$ at the identity of $E$. We will denote this fiber functor by $\omega_\vv$. For an admissible variation of Hodge structures, this amounts to taking the limit mixed Hodge structure associated to the tangent vector, see the natural definition given in \cite{drh}. See also \cite[\S 15.3 -- 15.12]{deligne:p1} for tangential base points, and \cite[\S 16]{kzb} for their relations to limit mixed Hodge structures. 
\end{remark}

\begin{corollary}\label{drfgs}
There is an isomorphism of groups over $\K$
$$\pi_1(\Ca^\DR_\K,\omega_x)\cong \exp\LL_\K(\Ss,\T)^\wedge$$
for any $x\in E(\K)$.
\end{corollary}

\begin{remark}
One can establish the isomorphism in a different way. By Deligne \cite[Cor. 10.43]{deligne:p1}, in the unipotent case, the tannakian fundamental groupoid is compatible with extension of scalars. In particular, for our category $\Ca^\DR_\K$, given any $x\in E(\K)$, restricting its tannakian fundamental groupoid to a diagonal point $(x,x)$ gives a tannakian fundamental group defined over $\K$, which is also compatible with extension of scalars, i.e. $\pi_1(\Ca^\DR_{\K})\times_\K\C\cong\pi_1(\Ca^\DR_\C)$. Therefore, one obtains an isomorphism
$$\pi_1(\Ca^\an,\omega_x)\to\pi_1(\Ca^\DR_\K,\omega_x)\times_\K\C.$$ 
\end{remark}

\subsection{Universal unipotent connection over an elliptic curve $E_{/\K}$}
Using the explicit universal connection $\nabla$ on $\Pbar_\DR$, we provide an explicit construction of the $\K$-connection that has regular singularity at the identity of $E_{/\K}$ on a unipotent vector bundle $\Vbar$ \footnote{One should think of this bundle as Deligne's canonical extension to $E$ of $\cV$ over $E'$.} in $\Ca^\DR_\K$ which, by Cor. \ref{drfgs}, corresponds to a unipotent representation
$$\rho:\exp\LL_\K(\Ss,\T)^\wedge\to\Aut(V).$$
This is achieved by composing the universal connection forms with the representation $\rho$.

Given a unipotent vector bundle $\Vbar$ over $E_{/\K}$ in $\Ca^\DR_\K$. Choose the fiber functor $\omega_\vv$ over the unit tangent vector $\vv$ at the identity of $E$ and denote by $V:=V_\vv$ the fiber over the tangent vector $\vv$ at the identity (see Remark \ref{ff}). This vector bundle $\Vbar$ corresponds to a unipotent representation
$$\rho:\exp\LL_\K(\Ss,\T)^\wedge\to\Aut(V),$$
and equivalently a Lie algebra homomorphism
$$\log\rho:\LL_\K(\Ss,\T)^\wedge\to\End(V).$$
Recall that we have defined 1-forms 
$$\nu_1^\alg\in\Omega^1(E')\otimes\LL_\K(\Ss,\T)^\wedge\quad\text{and}\quad \nu_1^\reg\in\Omega^1(E''\log\{\id\})\otimes\LL_\K(\Ss,\T)^\wedge$$
in Section \ref{algconn} and Section \ref{regconn}, respectively. They are gauge equivalent on $E'\cap E''$ via the transformation
$$g_\reg:E'\cap E''\to\exp\LL_\K(\Ss,\T)^\wedge\subset\Aut\LL_\K(\Ss,\T)^\wedge.$$
Define 1-forms
$$\Omega'_V:=(1\otimes\log\rho)\circ\nu_1^\alg\in\Omega^1(E')\otimes\End(V)$$
and
$$\Omega''_V:=(1\otimes\log\rho)\circ\nu_1^\reg\in\Omega^1(E''\log\{\id\})\otimes\End(V).$$ Over $E'$ and $E''$, we endow trivial bundles
$$
\begin{tikzcd}
V\times E'\arrow[d] &\text{and}& V\times E''\arrow[d]\\
E' && E''
\end{tikzcd}
$$ 
with connections $\nabla=d+\Omega'_V$ and $\nabla=d+\Omega''_V$, respectively. Define
$$g_V:E'\cap E''\to \Aut(V)$$
by $g_V:=\exp(\log\rho\circ\log g_\reg)$, then $\Omega'_V$ and $\Omega''_V$ are gauge equivalent on $E'\cap E''$ via $g_V$. After gluing these two trivial bundles by the gauge transformation $g_V$, we obtain a connection $\nabla$ on $\Vbar$ defined over $\K$ such that
$$\nabla:\Vbar\to\Vbar\otimes\Omega^1_E(\log\{\id\}).$$

\part{Universal Elliptic KZB Connection---Algebraic Formula}

Levin--Racinet \cite[\S 5]{levin-racinet} sketched a proof to show that the bundle $\cP$ over $\E$ and its connection, the universal elliptic KZB connection, are defined over $\Q$. However, just as in the case of a single elliptic curve, their work is incomplete in that their connection formula has irregular singularities along the identity section of $\E$.

We show that after an algebraic change of gauge, the universal elliptic KZB connection has regular singularities along the identity section of $\E$ and the nodal cubic. Since all these data are defined over $\Q$, we have completed the work.

Similar to the previous part, we compute explicitly the connection formula in terms of algebraic coordinates on $\E$. We resolve the issue of irregular singularities by trivializing the bundle on two open subsets $\E'$ and $\E''$ of $\E$, where $\E'$ is obtained from $\E$ by removing the identity section\footnote{It is $\M_{1,2}$ as defined in Section \ref{orbi}.}, and $\E''$ is obtained from $\E$ by removing three sections that correspond to three nontrivial order $2$ elements on each fiber. On both open subsets, the algebraic connection formulas are defined over $\Q$, and the one on $\E''$ has regular singularities along the identity section. Note that the singularities around the nodal cubic are regular on both open subsets, and the gauge transformation on their intersection is compatible with the canonical extension $\Pbar$ of $\cP$ over $\E$ to $\Ebar$. One can think of the universal elliptic KZB connection as an algebraic connection on an algebraic vector bundle $\Pbar_\DR$ over $\Ebar$, which is defined over $\Q$ with regular singularities along boundary divisors. Therefore, we have constructed a $\Q$-de Rham structure $\Pbar_\DR$ on $\Pbar$ over $\Ebar$.

\section{Algebraic Connection Formula over $\E'$}

In Section \ref{KZB}, we defined the universal elliptic KZB connection $\nabla^\an$ on the bundle $\cP$ over $\E'$ analytically. This bundle $\cP$ can be pulled back to a bundle over $\mathcal{M}^\an_{1,1+\vec{1}}$ with connection, which we also denote by $\nabla^\an$. In this section, we will write this connection in terms of algebraic coordinates $x,y,u,v$ on $\mathcal{M}_{1,1+\vec{1}}$ (defined in Section \ref{stack}). The connection is $\G_m$-invariant, and is trivial on each fiber of $\mathcal{M}_{1,1+\vec{1}}\rightarrow\M_{1,2}$, thus descends to a connection on $\M_{1,2}=\E'$.

As in the case of a single elliptic curve, fiber by fiber, we apply the gauge transformation of $g_\alg(\xi,\tau)=\exp(-\frac{E_1}{2\pi i}\T)$ with both $g_\alg$ and $E_1$ having the extra variable $\tau$. After the gauge transformation, the connection 
$$\nabla^\an=d+\omega$$
 transforms into
$$\nabla=d+\omega_{\alg}=d-dg_\alg\cdot g_\alg^{-1}+g_\alg\,\omega\, g_\alg^{-1}.$$
So using Formulas \ref{f2}, \ref{f1}, \ref{cor of f1} and Lemma 9.3 in \cite{kzb} we have
\begin{align*}
\omega_{\alg}
  &=-dg_\alg\cdot g_\alg^{-1}+ g_\alg\cdot \left(2\pi i\,d\tau\otimes\AA\frac{\partial}{\partial \T}\right)+g_\alg\,\psi\, g_\alg^{-1}+g_\alg\,\nu\, g_\alg^{-1} 
\\&=\frac{1}{2\pi i}(\frac{\partial E_1}{\partial\xi}d\xi+\frac{\partial E_1}{\partial\tau}d\tau)\T
	 +2\pi i\,d\tau\otimes\AA\frac{\partial}{\partial \T}+\psi
\\&\quad +\frac{1-\exp(-\frac{E_1}{2\pi i}\T)}{\T}\cdot\AA\otimes{2\pi i} d\tau+\exp(-\frac{E_1\T}{2\pi i})\T F^{\zag}(2\pi i\xi,\T,\tau)\cdot\AA\otimes d\xi	  
\\&\quad +\exp(-\frac{E_1\T}{2\pi i})\big(\frac{1}{\T}+\T\frac{\partial F^{\zag}}{\partial \T}(2\pi i\xi,\T,\tau)\big)\cdot\AA\otimes 2\pi i d\tau
\\&=\frac{1}{2\pi i}(-E_2d\xi+\frac{1}{2\pi i}(E_3-E_1E_2)d\tau)\T
	 +2\pi i\,d\tau\otimes\AA\frac{\partial}{\partial \T}+\psi
\\&\quad +\frac{1-\exp(-\frac{E_1}{2\pi i}\T)}{\T}\cdot\AA\otimes {2\pi i} d\tau+\exp\big(-\sum_{k=2}^\infty\frac{(-\T)^k}{k}P_k(\xi,\tau)\big)\cdot\AA\otimes 2\pi i d\xi	 
\\&\quad + 2\pi i d\tau\otimes\Big[\frac{\exp(-\frac{E_1}{2\pi i}\T)}{\T}
+\exp(-\sum_{k=2}^\infty\frac{(-\T)^k}{k}P_k(\xi,\tau))(\sum_{k=1}^\infty(-\T)^{k-1}P_k(\xi,\tau)-\frac{1}{\T})\Big]\cdot\AA
\\&= \Big(-(2\pi i)^{-2}(E_2-e_2)\T+\exp(-\sum_{k=2}^\infty\frac{(-\T)^k}{k}P_k(\xi,\tau))\cdot\Ss\Big)\otimes 2\pi i(d\xi+\frac{1}{2\pi i}E_1d\tau)
\\&\quad +(2\pi i)^{-3}E_3\T\otimes 2\pi i d\tau+2\pi i\,d\tau\otimes\AA\frac{\partial}{\partial \T}+\psi
\\&\quad +\Big[\exp(-\sum_{k=2}^\infty\frac{(-\T)^k}{k}P_k(\xi,\tau))(\sum_{k=2}^\infty(-\T)^{k-1}P_k(\xi,\tau)-\frac{1}{\T})+\frac{1}{\T}\Big]\cdot\Ss\otimes {2\pi i}d\tau
\end{align*}

Recall the map from Section \ref{orbi}
\begin{align*}
(\C\times\h)-\Lambda_\h &\to \M_{1,1+\vec{1}}^\an=\{(x,y,u,v)\in\C^2\times\C^2: y^2=4x^3-ux-v, (u,v)\neq (0,0) \}
\\ (\xi,\tau) &\mapsto (P_2(\xi,\tau), -2P_3(\xi,\tau),20\Eis_4(\tau),\frac 73\Eis_6(\tau))
\end{align*}
that induces an isomorphism $\SL_2(\Z) \ltimes \Z^2\dbs((\C\times\h)-\Lambda_\h)\cong\G_m\dbs\M_{1,1+\vec{1}}^\an$. By pulling back through this map, we identify some algebraic forms with their analytic counterparts appeared in the formula above in the following
\begin{lemma}\label{algform}
Set $\alpha=2udv-3vdu$, $\Delta=u^3-27v^2$. Then $2\pi i\,d\tau=\frac{3\alpha}{2\Delta}$, and $2\pi i(d\xi+\frac{1}{2\pi i}E_1d\tau)=\frac{dx}{y}-\frac{6x^2-u}{y}\frac{\alpha}{2\Delta}-\frac{1}{6}\frac{d\Delta}{\Delta}\frac{x}{y}$.
\end{lemma}
\begin{proof}
Direct computation from \cite[Prop. 5.2.3]{levin-racinet}.
\end{proof}

Recall from Remark \ref{recur} that $P_k(\xi,\tau)=(2\pi i)^{-k}(E_k-e_k)$, $k\ge 2$ can be written as rational polynomials of $x=P_2(\xi,\tau)$, $y=-2P_3(\xi,\tau)$, $u=20\Eis_4(\tau)$ and $v=\frac 73\Eis_6(\tau)$, i.e. for all $k\ge 2$, it can be written as $P_k(x,y,u,v)\in \mathbb{Q}[x,y,u,v]$. Combining this with Lemma \ref{algform}, we only need to show that in terms of basis elements $\T$ and $\Ss$,
$$d+2\pi i\,d\tau\otimes\AA\frac{\partial}{\partial \T}+\psi$$
is algebraic. But with respect to the above framing, $d+2\pi i\,d\tau\otimes\AA\frac{\partial}{\partial \T}$ transforms to (cf. \cite{kzb} Prop 19.6)
\begin{equation}
d+\left(-\frac{1}{12}\frac{d\Delta}{\Delta}\T+\frac{3\alpha}{2\Delta}\Ss\right)\frac{\partial}{\partial \T}+\left(-\frac{u\alpha}{8\Delta}\T+\frac{1}{12}\frac{d\Delta}{\Delta}\Ss\right)\frac{\partial}{\partial \Ss},
\end{equation}
and $\psi$ transforms to 
\begin{equation}
\sum_{m\ge 1}\Big(\frac{1}{(2m)!}\frac{3\alpha}{2\Delta}p_{2m+2}(u,v)\otimes\sum_{\substack{j+k=2m+1\\ j,k>0}}(-1)^j[\ad_{\T}^j(\Ss), \ad_{\T}^k(\Ss)]\frac{\partial}{\partial\Ss}\Big),
\end{equation}
where $\Eis_{2m+2}$ is replaced by $p_{2m+2}(u,v)\in\mathbb{Q}[u,v]$ ($p_{2m}(u,v)$'s are polynomials defined by $G_{2m}(\tau)=p_{2m}(20G_4(\tau),7G_6(\tau)/3)$, where $G_{2m}$ is a normalized Eisenstein series of weight $2m$), and $\ad^j_\T$ denotes the operation that takes the adjoint action of $\T$ repeatly for $j$ times. Note that every derivation $\delta\in\Der\,\LL_{\Q}(\Ss,\T)^\wedge$ can be written uniquely in the form
$$
\delta=\delta(\Ss)\frac{\partial}{\partial\Ss}+\delta(\T)\frac{\partial}{\partial\T},
$$
as it is determined by its values on $\Ss$ and $\T$.

So the algebraic 1-form of the universal elliptic KZB connection is given by
\begin{align*}
\omega_{\alg}=&\left(-\frac{1}{12}\frac{d\Delta}{\Delta}\T+\frac{3\alpha}{2\Delta}\Ss\right)\frac{\partial}{\partial \T}+\left(-\frac{u\alpha}{8\Delta}\T+\frac{1}{12}\frac{d\Delta}{\Delta}\Ss\right)\frac{\partial}{\partial \Ss}
\\&+\Big(-\frac{xdx}{y}+\frac{1}{4}\frac{\alpha}{\Delta}\frac{ux+3v}{y}+\frac{1}{6}\frac{d\Delta}{\Delta}\frac{x^2}{y}\Big)\T
\\&+\exp\big(-\sum_{k=2}^\infty\frac{(-\T)^k}{k}P_k(x,y,u,v)\big)\cdot \Ss\left(\frac{dx}{y}-\frac{6x^2-u}{y}\frac{\alpha}{2\Delta}-\frac{1}{6}\frac{d\Delta}{\Delta}\frac{x}{y}\right)
\\&+\Big[\exp\big(-\sum_{k=2}^\infty\frac{(-\T)^k}{k}P_k(x,y,u,v)\big)(\sum_{k=2}^\infty(-\T)^{k-1}P_k(x,y,u,v)-\frac{1}{\T})+\frac{1}{\T}\Big]\cdot \Ss\frac{3\alpha}{2\Delta}
\\&+\sum_{m\ge 1}\Big(\frac{1}{(2m)!}\frac{3\alpha}{2\Delta}p_{2m+2}(u,v)\otimes\sum_{\substack{j+k=2m+1\\ j,k>0}}(-1)^j[\ad_{\T}^j(\Ss), \ad_{\T}^k(\Ss)]\frac{\partial}{\partial\Ss}\Big)
\\=&\left(-\frac{1}{12}\frac{d\Delta}{\Delta}\T+\frac{3\alpha}{2\Delta}\Ss\right)\frac{\partial}{\partial \T}+\left(-\frac{u\alpha}{8\Delta}\T+\frac{1}{12}\frac{d\Delta}{\Delta}\Ss\right)\frac{\partial}{\partial \Ss}
\\&+\Big(-\frac{xdx}{y}+\frac{1}{4}\frac{\alpha}{\Delta}\frac{ux+3v}{y}+\frac{1}{6}\frac{d\Delta}{\Delta}\frac{x^2}{y}\Big)\T+\left(\frac{dx}{y}-\frac{6x^2-u}{y}\frac{\alpha}{2\Delta}-\frac{1}{6}\frac{d\Delta}{\Delta}\frac{x}{y}\right)\Ss
\\&+\sum_{n\ge 2}\left(\frac{dx}{y}-\frac{6x^2-u}{y}\frac{\alpha}{2\Delta}-\frac{1}{6}\frac{d\Delta}{\Delta}\frac{x}{y}+(n-1)\frac{3\alpha}{2\Delta}\right) q_n(x,y,u,v)\ \T^n\cdot \Ss
\\&+\sum_{m\ge 1}\Big(\frac{1}{(2m)!}\frac{3\alpha}{2\Delta}p_{2m+2}(u,v)\otimes\sum_{\substack{j+k=2m+1\\ j,k>0}}(-1)^j[\T^j\cdot\Ss, \T^k\cdot\Ss]\frac{\partial}{\partial \Ss}\Big)
\end{align*}
where $\Delta=u^3-27v^2$, $\alpha=2udv-3vdu$, $q_n(x,y,u,v)\in\Q[x,y,u,v]$ ($n\ge 2$) are essentially the same polynomials as in Section \ref{algconn} but with two more variables $u,v$ (previously $u,v$ are fixed as the elliptic curve is fixed) and $p_{2m}(u,v)\in\Q[u,v]$ are polynomials we just defined. This 1-form takes value in $\Der\,\LL_\Q(\Ss,\T)^\wedge$. We view $\LL_{\Q}(\Ss,\T)^\wedge$ as a Lie subalgebra of $\Der\,\LL_{\Q}(\Ss,\T)^\wedge$ via the adjoint action, and $\T^n$ acts on $\LL_{\Q}(\Ss,\T)^\wedge$ as $\ad_\T^n$.

\section{Algebraic Connection Formula over $\E''$}\label{QDR}

As in the single elliptic curve case, we apply the gauge transformation $g_\reg=\exp(-\frac{2x^2}{y}\T)$ to the previous formula for the algebraic 1-form, and obtain the algebraic 1-form:
\begin{align*}
\omega_\reg=&-dg_\reg\cdot g_\reg^{-1}+g_\reg\,\omega_\alg\, g_\reg^{-1}
\\=&\left(-\frac{1}{12}\frac{d\Delta}{\Delta}\T+\frac{3\alpha}{2\Delta}\Ss\right)\frac{\partial}{\partial \T}+\left(-\frac{u\alpha}{8\Delta}\T+\frac{1}{12}\frac{d\Delta}{\Delta}\Ss\right)\frac{\partial}{\partial \Ss}
\\&+\left[\left(d\left(\frac{2x^2}{y}\right)-\frac{xdx}{y}\right)+\frac{1}{4}\frac{\alpha}{\Delta}\frac{ux+3v}{y}\right]\T+\left(\frac{dx}{y}+\frac 1y\frac{u\alpha}{2\Delta}-\frac{1}{6}\frac{d\Delta}{\Delta}\frac{x}{y}\right)\Ss
\\&+\sum_{n\ge 1}\left(\frac{dx}{y}+\frac{1}{y}\frac{u\alpha}{2\Delta}-\frac{1}{6}\frac{d\Delta}{\Delta}\frac{x}{y}+(n-1)\frac{3\alpha}{2\Delta}\right)r_n(x,y,u,v)\T^n\cdot \Ss
\\&+\sum_{m\ge 1}\Big(\frac{1}{(2m)!}\frac{3\alpha}{2\Delta}p_{2m+2}(u,v)\otimes\sum_{\substack{j+k=2m+1\\ j,k>0}}(-1)^j[\T^j\cdot\Ss, \T^k\cdot\Ss]\frac{\partial}{\partial \Ss}\Big)
\end{align*}
where $r_n(x,y,u,v)\in\Q(x,y,u,v)$ ($n\ge 2$) are essentially the same rational functions as in Section \ref{regconn} but with two more variables $u,v$. 

Note that the $\mathbb{G}_m$-action of $\lambda$ multiplies $\T$ by $\lambda$, and $\Ss$ by $\lambda^{-1}$. It is easy to check that both connection forms $\omega_\alg$ and $\omega_\reg$ are $\mathbb G_m$-invariant. One can also show that the latter connection form $\omega_\reg$ has regular singularity along the identity section, and along the nodal cubic; the residue of the connection around the identity section is $\text{ad}_{[\T,\Ss]}$, which is pronilpotent. 

Just like the single elliptic curve case, we can use both connections $\omega_\alg$ and $\omega_\reg$ with the gauge tranformation $g_\reg$ between them to construct a vector bundle $\cP_\DR$ over $\E_{/\Q}$. Since both connection forms are defined over $\Q$ and have regular singularities along the nodal cubic, we can extend $\cP_\DR$ to $\Ebar_{/\Q}$ and obtain an algebraic vector bundle $\Pbar_\DR$. 

Let $(\Pbar,\nabla^\an)$ be Deligne's canonical extension to $\Ebar$ of the bundle $\cP$ of (Lie algebras of) unipotent fundamental groups over $\E'$. We have
\begin{theorem}[\bf The $\Q$-de Rham structure $\Pbar_\DR$ on $\Pbar$ over $\Ebar$]
There is an algebraic vector bundle $\Pbar_\DR$ over $\Ebar_{/\Q}$ endowed with connection $\nabla$, and an isomorphism $$(\Pbar_\DR,\nabla)\otimes_{\Q}\C\approx(\Pbar,\nabla^\an).$$
The algebraic bundle $\Pbar_\DR$ and its connection $\nabla$ are both defined over $\Q$. The $\Q$-de Rham structure $(\Pbar_\DR,\nabla)$ on $(\Pbar,\nabla^\an)$ is explicitly given by the connection formulas for $\omega_\alg$ on $\E'$ and $\omega_\reg$ on $\E''$ above. In particular, the connection $\nabla$ has regular singularities along boundary divisors, the identity section and the nodal cubic.
\end{theorem}

\end{document}